\documentclass[12pt,reqno]{amsart}
\usepackage{cite}

\usepackage{verbatim,amscd,amssymb}
\usepackage{graphicx}
\usepackage{amsmath}
	\usepackage{amsmath}
\usepackage{amsfonts}
\usepackage{amssymb}
\usepackage{inputenc}
\usepackage{enumitem}
\usepackage{float}

\usepackage[active]{srcltx}

\graphicspath{ {./images/} }

\usepackage{fontenc}

\usepackage{float}

\newtheorem{theorem}{Theorem}[section]
\newtheorem{lemma}[theorem]{Lemma}

\theoremstyle{definition}
\newtheorem{definition}[theorem]{Definition}

\theoremstyle{remark}

\theoremstyle{assumption}

\numberwithin{equation}{section}

\newcommand{\sha}{\succ\mkern-14mu_s\;}

\begin{document}
	
	\date{\today}

\title[Triod twist cycles and circle rotations]{Triod twist cycles and circle rotations}

\author{Sourav Bhattacharya and Ashish Yadav}

\address[Dr. Sourav Bhattacharya and Ashish Yadav]
{Department of Mathematics, Visvesvaraya National Institute of Technology Nagpur,
	Nagpur, Maharashtra 440010,
	India}
	\email{souravbhattacharya@mth.vnit.ac.in}


\begin{abstract}
	
We study the problem of relating cycles on  a \emph{triod} $Y$ to \emph{circle rotations}.  We prove that the simplest cycles on a \emph{triod}~$Y$ with a given \emph{rotation number}~$\rho$, called \emph{triod--twist cycles} are conjugate, via a piece-wise monotone map of \emph{modality} at most~$m + 3$, where~$m$ is the \emph{modality} of~$P$  to the rotation on~$S^1$ by angle~$\rho$, restricted to one of its cycles.

\end{abstract}

	\maketitle

\section{Introduction}\label{intro}

In many branches of mathematics, a fundamental objective is the classification of the objects under investigation. Such classifications are typically carried out by means of \emph{invariants}, that is, properties that remain unchanged under isomorphisms. In the setting of topological dynamical systems, the role of an isomorphism is played by  \emph{topological conjugacy}. Consequently, the classification problem reduces to identifying quantities that are preserved under topological conjugacies. In the case of interval maps, one may further impose the natural requirement that the conjugacy preserve \emph{orientation}. Remarkably, when the period of a periodic orbit (also called a cycle) is regarded as such an \emph{invariant}, the resulting theory exhibits a rich and intriguing structure.

 In 1964,  A.~N.~Sharkovsky~\cite{shatr} formulated a celebrated Theorem which gives a complete description of all possible sets of periods of periodic orbits of a continuous interval map. He introduced a total order on the set of  natural numbers, $\mathbb{N}$ known as the \emph{Sharkovsky ordering}: 
$$3\sha 5\sha 7\sha\dots\sha 2\cdot3\sha 2\cdot5\sha 2\cdot7\sha\dots$$
$$\sha\dots\sha 2^2\cdot3\sha 2^2\cdot5\sha 2^2\cdot7\sha\dots\sha 8\sha 4\sha 2\sha 1.$$

For $k\in\mathbb{N}$, let $Sh(k)$ denote the set of all integers $m$ satisfying $k\sha m$, including $k$ itself.  
Define $Sh(2^\infty)=\{1,2,4,8,\dots\}$, the set of all powers of $2$.  
Let $Per(f)$ denote the set of periods of cycles of a continuous map $f$.  Sharkovsky proved the following:

\begin{theorem}[\cite{shatr}]\label{t:shar}
	If $f:[0,1]\to[0,1]$ is continuous, $m\sha n$ and $m\in Per(f)$, then $n\in Per(f)$.  
	Consequently, there exists $k\in \mathbb{N}\cup\{2^\infty\}$ such that $Per(f)=Sh(k)$.  
	Conversely, for any $k\in \mathbb{N}\cup\{2^\infty\}$, there exists a continuous map $f:[0,1]\to[0,1]$ with $Per(f)=Sh(k)$.
\end{theorem}

Theorem \ref{t:shar} initiated a new field of research called \emph{combinatorial one-dimensional dynamics}.  Subsequently, research has progressed in many directions, including extensions of Theorem \ref{t:shar} to more general spaces, development of refined coexistence rules, and investigations involving more complex invariant objects such as the \emph{homoclinic trajectories}.

The present work is primarily motivated by the first two of these
directions. Once one moves beyond the interval, the dynamical picture
becomes considerably more intricate. Even the simplest graph-like
spaces introduce \emph{branching}, thereby destroying the linear order
that underpins much of classical interval dynamics. Among
such spaces, the \emph{triods} provides the most basic example of a
\emph{branched one-dimensional continuum}. Consequently, a detailed understanding of
periodic orbits on \emph{triods} is not only a natural extension of interval dynamics,  but a  natural and necessary step toward a broader theory of dynamics on graphs and related spaces. The
main objective of this paper is to investigate how periodic orbits of
maps on the \emph{triod} $Y$ relate to the simpler and well-understood model of
\emph{circle rotations}, restricted to their periodic orbits.

An \emph{$n$-od} (or \emph{$n$-star}) $X_n$ is defined as the set of complex numbers $z\in\mathbb{C}$ for which $z^n\in[0,1]$.  
Geometrically, it is the union of $n$ copies of the interval $[0,1]$, attached at a single common point called the \emph{branching point}, denoted by $a$.  
Each component of $X_n \setminus\{a\}$ is called a \emph{branch} of $X_n$. $X_2$ is homeomorphic to an interval, while $X_3$ is called a \emph{triod} and denoted by $Y$.  
Dynamics on $X_n$ have received considerable attention not only due to their intrinsic mathematical appeal, but also because they arise naturally as quotient models for higher-dimensional systems with \emph{invariant foliations} or \emph{surface homeomorphisms}.  They also appear in complex dynamics---for instance, in the study of \emph{Hubbard trees}. 

The problem of describing the possible periods of cycles for a continuous map $f:Y\to Y$ fixing the \emph{branching point} $a$ was first considered in \cite{alm98,Ba}, and later refined in \cite{almnew}.  
It was observed there that the resulting sets of periods form unions of ``initial segments'' of certain linear orders on rational numbers in $(0,1)$ with denominator at most $3$.  
However, this phenomenon arose from computational evidence and lacked a theoretical explanation.

A breakthrough came in 2001, when Blokh and Misiurewicz~\cite{BMR} introduced  \emph{rotation theory} for maps on \emph{triods}, providing an explanation to the earlier observations. It is worth emphasizing that the notion of a \emph{rotation number} traces its origins to the pioneering work of Poincar\'e in his study of \emph{circle homeomorphisms} (see~\cite{poi}).

We now summarize \emph{rotation theory} for \emph{triods} as developed in \cite{BMR}.  
For points $x,y\in Y$, write $x>y$ if they lie on the same \emph{branch} of $Y$ and $x$ is farther from the \emph{branching point} $a$ than $y$, and write $x\geqslant  y$ if $x>y$ or $x=y$.  
For $A\subset Y$, denote its \emph{convex hull} by $[A]$.

Two cycles $P$ and $Q$ in $Y$ are said to be \emph{equivalent} if there exists a homeomorphism $h:[P]\to[Q]$ that conjugates the dynamics and fixes each \emph{branch} of $Y$.   
The equivalence classes of cycles under this relation are called \emph{patterns}.  
If $f\in\mathcal{U}$ and $P$ is a cycle of $f$ belonging to the equivalence class $A$, we say that $P$ \emph{is a representative of} (or \emph{exhibits})  \emph{pattern} $A$.

A cycle (and hence its \emph{pattern}) is called \emph{primitive} if all points of the cycle lie on distinct \emph{branches} of $Y$. A \emph{pattern} $A$ is said to \emph{force} a \emph{pattern} $B$ if every map $f \in \mathcal{U}$ having a cycle of \emph{pattern} $A$ must also have a cycle of \emph{pattern} $B$. It follows from \cite{alm98,BMR} that if $A$ \emph{forces} $B$ with $A \neq B$, then $B$ \emph{does not force} $A$. Similarly, a cycle $P$ is said to \emph{force} a cycle $Q$ if the \emph{pattern} of $P$ forces the \emph{pattern} of $Q$.

	A map $f $ is called $P$-\emph{linear} for a cycle $P$ if it fixes $a$, is \emph{affine} on every component of $[P] - (P \cup \{ a\})$ and also constant on every component of $Y - [P]$ where $[P]$ is the convex hull of $P$.  The next result gives an effective criterion for determining which \emph{patterns} are forced by a given one.

\begin{theorem}[\cite{alm98,BMR}]\label{forcing}
	Let $f$ be a $P$-linear map, where $P$ is a cycle exhibiting pattern $A$.
	Then a pattern $B$ is forced by $A$ if and only if $f$ has a cycle $Q$ exhibiting pattern $B$.
\end{theorem}

We now state \emph{rotation theory} for \emph{triods} as formulated in \cite{BMR}.
We consider $Y$ embedded in the plane with the \emph{branching point} at the origin, and its \emph{branches} being straight-line segments. Label the \emph{branches} of $Y$ in clockwise order as $B = \{ b_i \mid i = 0,1,2 \}$, where indices are taken modulo $3$.  Let $ f \in \mathcal{U}$ and  $P \subset Y- \{a \}$  be finite. By an \emph{oriented graph} corresponding to $P$,  we shall mean a graph $G_P$ whose vertices are elements of $P$ and arrows are defined as follows. For $x,y \in P$, we will say that there is an arrow from $x$ to $y$ and write $x \to y$ if there exists $ z \in Y$ such that $ x \geqslant z$ and $ f(z) \geqslant y$ (See Figure \ref{directed edge}). We will refer to a \emph{loop}  in the \emph{oriented graph} $G_P$ as a \emph{``point" loop} in $Y$. We call a \emph{point loop} in $Y$ \emph{elementary} if it passes through every vertex of $G_P$ at most once. If $P$ is a cycle of period $n$, then the \emph{loop} $\Gamma_P : x \to f(x) \to f^2(x) \to f^3(x) \to \dots f^{n-1}(x) \to x$, $x \in P$ is called the \emph{fundamental point loop} associated with $P$

\begin{figure}[H]
	\centering
	\includegraphics[width=0.4\textwidth]{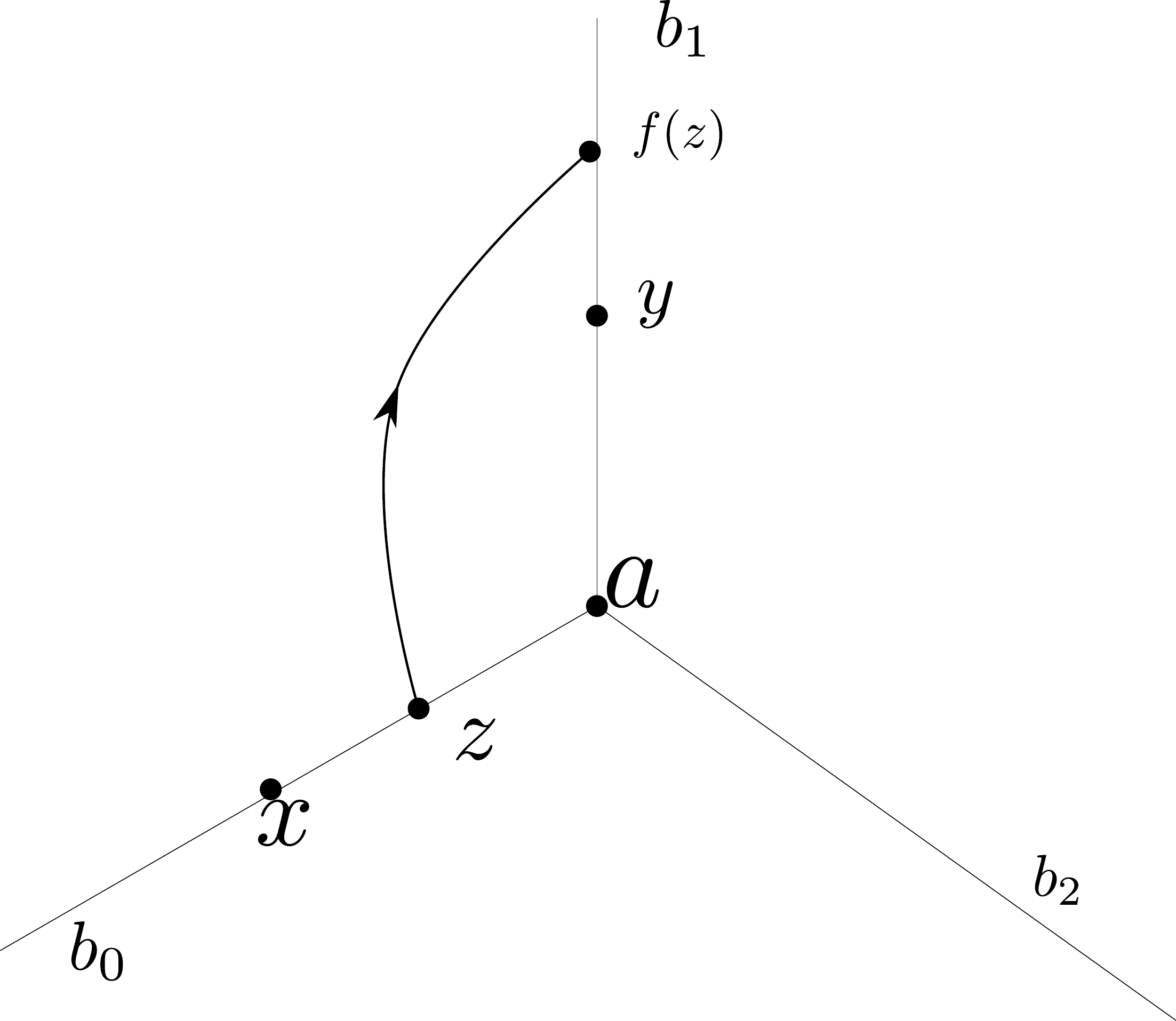}
	\caption{\emph{Directed edge} from $x$ to $y$}
	\label{directed edge}
\end{figure}

Let us assume that the \emph{oriented graph} $G_P$ corresponding to $P$ is \emph{transitive} (that is,  there is a \emph{path} in $G_P$ from every vertex to every vertex). Observe that, if $P$ is a cycle, then $G_P$ is always transitive.  Let $A$ be the set of all arrows of the \emph{oriented graph} $G_P$. We define a \emph{displacement function} $ d : A \to \mathbb{R}$  by $d(u \to v) = \frac{k}{3}$ where $ u \in b_i$ and $v \in b_j$ and $ j = i +k $ (modulo 3).  For a \emph{point} \emph{loop} $\Gamma$ in $G_P$ denote by $ d(\Gamma)$ the sum of the values of the \emph{displacement} $d$ along the \emph{loop}. In our model of $Y$, this number tells us how many times we revolved around the origin in the clockwise sense. Thus, $ d(\Gamma)$ is an integer. We call  $ rp(\Gamma) = (d(\Gamma), |\Gamma|)$  and $ \rho(\Gamma) = \frac{d(\Gamma)}{ |\Gamma|}$, the \emph{rotation pair} and \emph{rotation number} of the \emph{point loop}, $\Gamma$  respectively (where $|\Gamma|$ denotes length of $\Gamma$).

The closure of the set of \emph{rotation numbers} of all \emph{loops} of the \emph{oriented graph},  $G_P$ is called the \emph{rotation set} of $G_P$ and denoted by $L(G_P)$. By \cite{zie95}, $L(G_P)$ is equal to the smallest interval containing the \emph{rotation numbers} of all \emph{elementary loops} of $G_P$. Following the notation in~\cite{BMR}, a \emph{rotation pair} $rp(\Gamma) = (mp, mq)$, where $p, q, m \in \mathbb{N}$ and $\gcd(p, q) = 1$, can be represented equivalently as $mrp(\Gamma) = (t, m)$, where $t = p/q$.  
The pair $(t, m)$ is called the \emph{modified rotation pair} (\emph{mrp}) associated with the \emph{loop} $\Gamma$.  For a periodic orbit $P$, the \emph{rotation number}, \emph{rotation pair}, and \emph{modified rotation pair} are defined to be those of its \emph{fundamental point loop} $\Gamma_P$.  
Similarly, for a given \emph{pattern} $A$, these quantities are determined from any cycle $P$ that \emph{exhibits} $A$.  
The \emph{rotation interval forced} by a pattern $A$ is the \emph{rotation set} $L(G_P)$ of the \emph{oriented graph} $G_P$ associated with such a cycle.  
We denote by $mrp(A)$ the set of all \emph{modified rotation pairs} corresponding to \emph{patterns} that are \emph{forced} by $A$.  

\begin{figure}[H]
	\centering
	\includegraphics[width=0.4\textwidth]{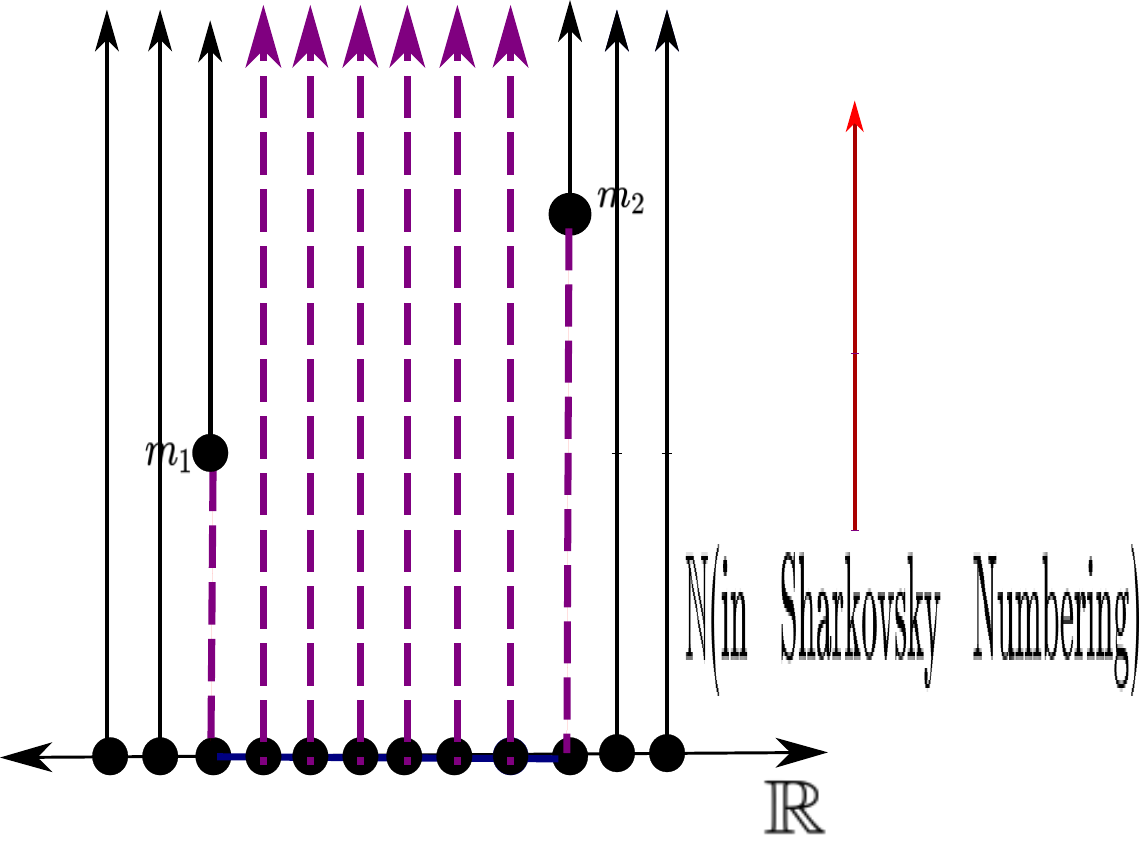}
	\caption{Graphical illustration of \emph{modified rotation pairs} (\emph{mrp}) on the real line with attached prongs.}
	\label{convex_hull}
\end{figure}

The notion of \emph{modified rotation pairs} has a convenient geometric representation (see Figure~\ref{convex_hull}).  
Consider the real line with a \emph{prong} attached at each rational point, while irrational points are associated with degenerate prongs.  
On each prong, the set $\mathbb{N} \cup \{2^{\infty}\}$ is arranged following the \emph{Sharkovsky ordering} $\sha$, placing $1$ closest to the real line and $3$ farthest from it.  
Points that lie directly on the real line are labeled as $0$.  
The union of the real line and its prongs is denoted by $\mathbb{M}$.  
Each \emph{modified rotation pair} $(t, m)$ corresponds to the point on $\mathbb{M}$ that represents the value $m$ on the prong attached at $t$.  
No actual \emph{rotation pair} corresponds to $(t, 2^{\infty})$ or $(t, 0)$.  For two points $(t_1, m_1)$ and $(t_2, m_2)$ in $\mathbb{M}$, their \emph{convex hull} is defined as  
$[(t_1, m_1), (t_2, m_2)] = \{ (t, m) \mid t_1 < t < t_2, \text{ or } (t = t_i \text{ and } m \in Sh(m_i)), \, i = 1, 2 \}.$

\begin{definition}\label{regular:defn}
	A \emph{pattern} $A$ is called \emph{regular} if it does not \emph{force} any \emph{primitive pattern} of period $2$.  
	A cycle $P$ is termed \emph{regular} if it exhibits a \emph{regular pattern}.
\end{definition}

Due to the transitivity of the forcing relation, any \emph{pattern} forced by a \emph{regular} one must itself be \emph{regular}. A map $f \in \mathcal{U}$ is called \emph{regular} if all of its cycles are \emph{regular}, and we denote the family of all such maps by $\mathcal{R}$.
\begin{theorem}[\cite{BM2}]\label{result:1}
	Let $A$ be a regular pattern associated with a map $f \in \mathcal{R}$.
	Then there exist patterns $B$ and $C$ with modified rotation pairs $(t_1 , m_1)$ and $(t_2 , m_2)$ such that
	$mrp(A) = [(t_1, m_1),(t_2, m_2)].$
\end{theorem}
Theorem~\ref{result:1} provides a complete description of the \emph{modified rotation pairs} arising from cycles of \emph{regular maps} defined on a \emph{triod}. In particular, it serves as a useful tool in deducing the dynamics of \emph{regular maps} $f \in \mathcal{R}$ by relying solely on minimal combinatorial data. 

The simplest \emph{patterns} on \emph{triods} corresponding to a prescribed \emph{rotation number} $\rho$, known as \emph{triod-twists}, were introduced in \cite{BB5}.
\begin{definition}
	A \emph{regular pattern} $A$ is called a \emph{triod-twist pattern} if it does not \emph{force} any other pattern having the same \emph{rotation number}.
\end{definition}
A cycle $P$ on the \emph{triod} $Y$ is called a \emph{triod-twist cycle} if it \emph{exhibits} a \emph{triod-twist pattern}. Equivalently, a cycle~$P$ on $Y$ is a \emph{triod-twist cycle} if there exists a map $f:Y \to Y$ for which $P$ is the \emph{unique} cycle with \emph{rotation number} $\rho$.
In~\cite{BB5}, such cycles were thoroughly studied and fully classified. Additionally, the dynamics of all \emph{unimodal triod-twist} cycles associated with a fixed \emph{rotation number} were described.

This naturally raises the following question: Does a \emph{triod-twist} cycle with \emph{rotation number}~$\rho$ admit a meaningful connection to a \emph{rotation} of the circle by angle~$\rho$, restricted to one of its cycles? More specifically, can one establish an explicit dynamical correspondence between the two? In other words, does the presence of \emph{branching} fundamentally alter the \emph{rotational nature} of these \emph{minimal cycles}, or does \emph{rotational rigidity} persist even in this more complex setting? The main purpose of this paper is to answer this question affirmatively, in case of \emph{triods}.

We show that  \emph{triod–twist cycles} are \emph{topologically conjugate} to  rotation by angle $\rho$ on $S^1$, restricted on one of its cycles,  under a piecewise monotone map of \emph{modality} at most $m+3$, where $m$ denotes the \emph{modality} of the cycle~$P$. Our result generalizes the analogous result obtained by Blokh and Misiurewicz for interval maps (see \cite{BM2}), and shows that \emph{rotational rigidity} persists despite \emph{branching}. Moreover, it shows that the \emph{modality} of the cycle alone yields explicit quantitative control over the complexity of the \emph{conjugacy}.

The paper is organized as follows. Section~2 introduces the necessary
preliminaries and notation. Section~3 contains the proofs of the main
results.
	
		\section{Preliminaries}\label{preliminaries}

	\subsection{Monotonicity}\label{monotonicity}
	
	A continuous map $f:Y \to Y$ is said to be \emph{monotone} on a subset $U \subset Y$ if the pre-image of every point $v \in f(U)$ is a \emph{connected} subset of $U$.  
	A subset $U \subset Y$ is called a \emph{lap} of $f$ if it is a \emph{maximal open set} on which $f$ is \emph{monotone}, maximality being understood with respect to inclusion.  
	The total number of \emph{laps} of a map $f \in \mathcal{U}$ is called its \emph{modality}.  
	For a periodic orbit (cycle) $P$, the \emph{modality of $P$} is defined as the \emph{modality} of the associated $P$-\emph{linear} map.
	
	\subsection{Loops of Points}\label{loops}
	
	In Section~\ref{intro}, we introduced the \emph{oriented graph} $G_P$ associated with a finite subset $P \subset Y \setminus \{a\}$ and defined the notion of \emph{point loops}.  
	The result below shows that analyzing \emph{point loops} in $G_P$—when a cycle $P$ \emph{exhibits} a \emph{pattern} $A$—is sufficient for determining all \emph{patterns} that are \emph{forced} by $A$.
	
	\begin{theorem}[\cite{BMR}]\label{loops:orbits:connection:1}
		The following statements hold:
		\begin{enumerate}
			\item For any point loop $x_0 \to x_1 \to \dots \to x_{m-1} \to x_0$ in $Y$, there exists a point $y \in Y \setminus \{a\}$ such that $f^m(y)=y$, and for each $k=0,1,\dots,m-1$, the points $x_k$ and $f^k(y)$ lie on the same branch of $Y$.
			\item If $f$ is a $P$-linear map corresponding to a nontrivial cycle $P \neq \{a\}$ and $y\neq a$ is a periodic point of $f$ of period $q$, then there exists a point loop  
			$x_0 \to x_1 \to \dots \to x_{q-1} \to x_0$ in $Y$ satisfying $x_i \geqslant f^i(y)$ for all $i$.
		\end{enumerate}
	\end{theorem}
	
	\subsection{Colors of Points}\label{colors}
	
	To encode the arrows in the oriented graph $G_P$ for $P \subset Y \setminus \{a\}$ finite, we use the \emph{color convention} from~\cite{BMR}.  
	For a \emph{directed edge} $u \to v$ in $G_P$, the \emph{color} is assigned depending on the \emph{displacement} $d(u \to v)$:  
	\emph{green} if $d(u \to v)=0$, \emph{black} if $d(u \to v)=\tfrac{1}{3}$, and \emph{red} if $d(u \to v)=\tfrac{2}{3}$.  
	If $P$ is a cycle of a map $f \in \mathcal{R}$, the \emph{color} of a point $x \in P$ is defined as the \emph{color} of the arrow $x \to f(x)$ in the \emph{fundamental point loop} $\Gamma_P$.  
	A \emph{loop} containing only \emph{black arrows} will be referred to as a \emph{black loop}.
	
	\begin{theorem}[\cite{BMR}]\label{black:loop:length:3}
		Let $f \in \mathcal{R}$ and let $P$ be a cycle of $f$. Then:
		\begin{enumerate}
			\item Every point $x \in P$ lies on a black loop of length $3$.
			\item If $x$ is green, then $x > f(x)$.
			\item The cycle $P$ has at least one point on each branch of $Y$.
			\item The cycle $P$ necessarily forces a primitive $3$-cycle.
		\end{enumerate}
		Moreover, there exists an ordering $\{b_i \mid i=0,1,2\}$ (indices taken mod~$3$) of the branches of $Y$ such that the points $p_i \in P$, $i=0,1,2$, that lie closest to the branch point $a$ on each branch $b_i$, are black.  
		This ordering is referred to as the canonical ordering of the branches of $Y$.
	\end{theorem}
	
	From now on, we label the \emph{branches} of $Y$ according to this \emph{canonical ordering}.

\subsection{Triod-twist patterns}

In~\cite{BB5}, it was shown that \emph{triod-twist patterns} undergo a qualitative change in their \emph{color} distribution at the critical \emph{rotation number} $\rho = \tfrac{1}{3}$ (see Figure~\ref{drawing2}).

 \begin{theorem}[\cite{BB5}]\label{bifurcation:one:third}
 	Let $\pi$ be a triod--twist pattern with rotation number $\rho$, and let $P$ be a periodic orbit that exhibits $\pi$. Then the following hold:
 	\begin{enumerate}
 		\item If $\rho < \tfrac{1}{3}$, $P$ contains green and black points, but no red points.
 		\item If $\rho = \tfrac{1}{3}$, $P$ is the primitive cycle of period $3$, and consequently all its points are black.
 		\item If $\rho > \tfrac{1}{3}$, $P$ contains red and black points, but no green points.
 	\end{enumerate}
 \end{theorem}

\begin{figure}[H]
	\centering
	\includegraphics[width=0.45\textwidth]{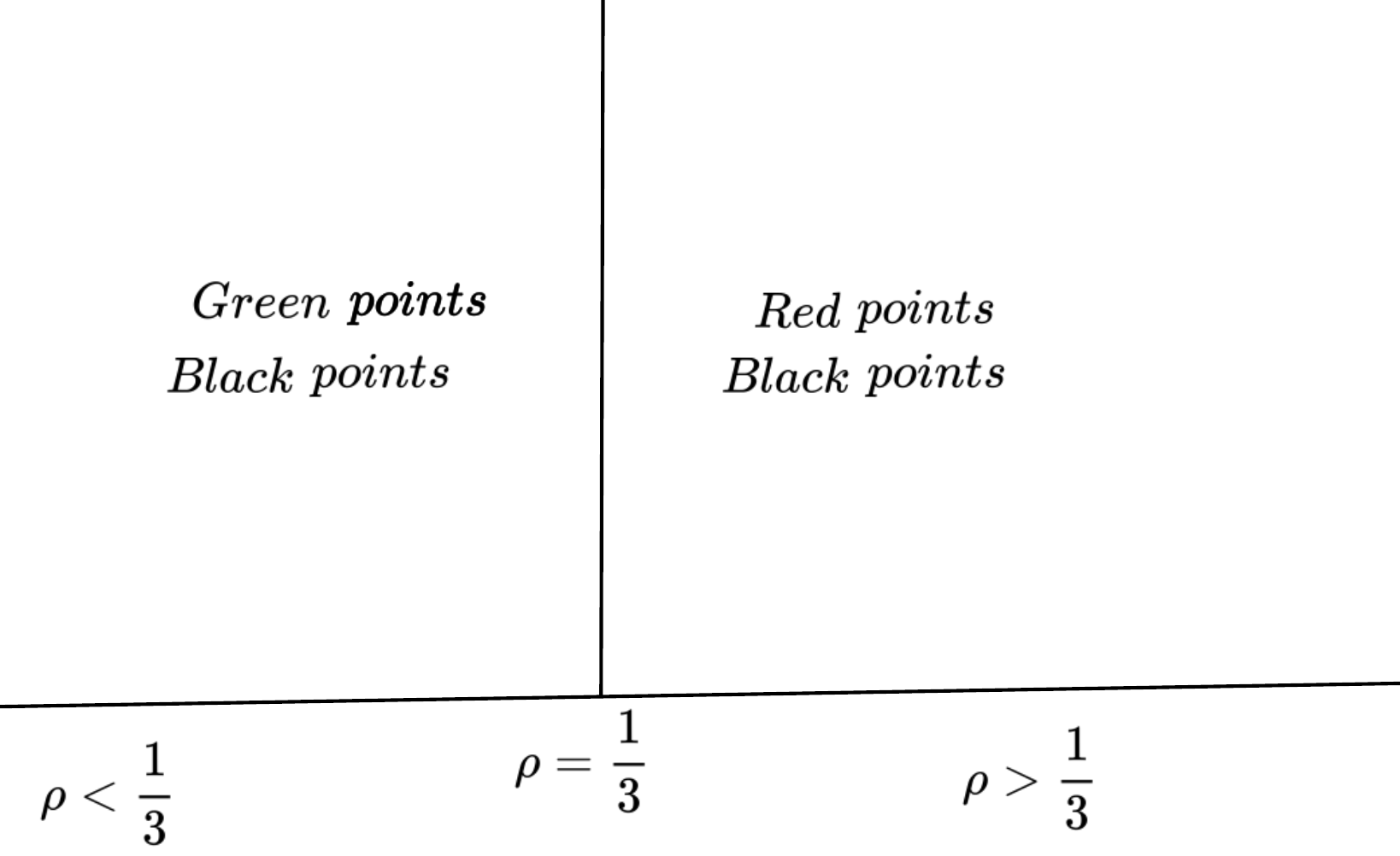}
	\caption{\emph{Bifurcation diagram} showing the variation in point colors with respect to the rotation number $\rho$.}
	\label{drawing2}
\end{figure}

We next state a necessary condition for a \emph{pattern} to be a \emph{triod-twist pattern}.

\begin{definition}[\cite{BB5}]\label{green}
	A \emph{regular cycle} $P$ is called \emph{order-preserving} if for any $x,y \in P$ with $x >  y$, whenever $f(x)$ and $f(y)$ lie on the same \emph{branch} of $Y$, it follows that $f(x)>  f(y)$.  
	A \emph{pattern} $A$ is called \emph{order-preserving} if every cycle \emph{exhibiting} $A$ is \emph{order-preserving}.
\end{definition}

\begin{theorem}[\cite{BB5}]\label{necesary:condition:triod:twist}
	Every triod-twist pattern is order-preserving.
\end{theorem}

\subsection{Code Function}\label{section:code}

The \emph{code function} was first introduced by Blokh and Misiurewicz for interval maps (see~\cite{BM2}) and later extended to \emph{triod} dynamics in~\cite{BB5}.  
It provides a combinatorial encoding of the ordering of points along a periodic orbit and relates this structure to the corresponding \emph{rotation number}.

\begin{definition}[\cite{BB5}]\label{code:function}
	Let $f \in \mathcal{R}$ have a periodic orbit $P$ with \emph{rotation number} $\rho \neq \tfrac{1}{3}$.  
	The \emph{code function} $L:P \to \mathbb{R}$ is defined as follows:
	\begin{enumerate}
		\item Fix an initial point $x_0 \in P$ and set $L(x_0) = 0$.
		\item For $k \geqslant 1$, define recursively
		\[
		L(f^k(x_0)) = L(x_0) + k\rho - [t_k],
		\]
		where $t_k = \sum_{j=0}^{k-1} d(f^j(x_0),f^{j+1}(x_0))$ and $[x]$ denotes the integer part of $x$.
	\end{enumerate}
\end{definition}

The value $L(x)$ is called the \emph{code} of the point $x \in P$.  
If $P$ has period $q$, then $t_q = q\rho \in \mathbb{Z}_+$, and thus $L(f^q(x_0)) = L(x_0)$, ensuring that $L$ is well defined on $P$.  
Moreover, for any $x,y \in P$, the difference $L(x)-L(y)$ does not depend on the particular choice of $x_0$.

\begin{definition}[\cite{BB5}]\label{non:decreasing}
	Let $L$ be the \emph{code function} for a cycle $P$ with \emph{rotation number} $\rho$.  
	We say that $L$ is \emph{non-decreasing} if for $x, y \in P$, 
	\begin{enumerate}
		\item If $\rho \leqslant \tfrac{1}{3}$, then $L(x) \leqslant L(y)$ whenever $x > y$.
		\item If $\rho >  \tfrac{1}{3}$, then $L(x) \geqslant L(y)$ whenever $x >  y$.
	\end{enumerate}
	Otherwise, $L$ is called \emph{decreasing}.  
	The function $L$ is \emph{strictly increasing} if it is \emph{non-decreasing} and no two consecutive points of $P$ within the same \emph{branch} of $Y$ have equal \emph{code} values.
\end{definition}

A \emph{pattern} $A$ is said to admit a \emph{non-decreasing} (respectively, \emph{strictly increasing}) \emph{code function} if every periodic orbit \emph{exhibiting} $A$ has such a \emph{code} function.  
The next theorem gives a complete characterization of \emph{triod-twist patterns} using this concept.

\begin{theorem}[\cite{BB5}]\label{tri-od:rot:twist:order:inv}
	A regular pattern $\pi$ is a triod-twist pattern if and only if it possesses a strictly increasing code function.
\end{theorem}

\section{Main Section}\label{bound:phase:function}
We will use the following notations throughout this section. For any subset $A \subseteq Y$, we denote by $i(A)$ and $e(A)$ the elements of $A$ that are respectively the \emph{nearest} and \emph{farthest} from the \emph{branch point} $a$.  
Furthermore, let $\chi(A)$ represent the \emph{supremum} of the set $\{L(x) - L(y) : x, y \in A\}$.  
For two subsets $A, B \subseteq Y$, we write $A \geqslant B$ if $a \geqslant b$ for every $a \in A$ and $b \in B$.  
Similarly, for a point $a \in Y$, the notation $a \geqslant B$ indicates that $a \geqslant b$ for all $b \in B$.

Let $\pi$ be a \emph{triod–twist pattern} with \emph{modality} $m$ and
\emph{rotation number} $\rho$. Suppose that $P$ is a cycle that
\emph{exhibits}  the \emph{pattern} $\pi$, and let $f$ be the associated
$P$–\emph{linear} map. Recall that the \emph{color} notations were introduced in Section \ref{colors}.

 We introduce an equivalence relation $\sim$ on $P$ as follows: for 
$a,b \in P$, we write $a \sim b$ if $a$ and $b$ have the same \emph{color}, 
and there is no point of $P$ of a different \emph{color} lying between them.  The equivalence classes of this relation are called \emph{states}.

 A \emph{state} consisting solely of \emph{green} points will be called a
 \emph{green state}; likewise, a \emph{state} made up entirely of \emph{black}
 points is a \emph{black state}, and one containing only \emph{red} points
 is a \emph{red state}. We denote the collections of all \emph{green}, \emph{black},
 and \emph{red states} of $P$ by $\mathcal{G}(P)$, $\mathcal{B}(P)$, and
 $\mathcal{R}(P)$, respectively. Moreover, we denote by $G(P)$, $B(P)$, and $R(P)$ the sets of all
 \emph{green}, \emph{black}, and \emph{red} points of $P$, respectively;
 that is, $G(P)$, $B(P)$, and $R(P)$ are precisely the unions of the
 elements of the collections $\mathcal{G}(P)$, $\mathcal{B}(P)$, and
 $\mathcal{R}(P)$.

\begin{lemma}\label{first:lemma}
 
The  cardinality of each of the sets $\mathcal{G}(P)$ and  $\mathcal{R}(P)$ is strictly less than $\frac{m}{2}+2$, $m$ being the modality of $f$. 
\end{lemma}

\begin{proof}
	Since $f$ has \emph{modality} $m$, the number of pre-images of the \emph{branching point} $a$ under $f$ can be at-most $m+1$. 	Whenever we have two \emph{states} $A$ and $B$ of different \emph{colors}, there exists a point of $f^{-1}(a)$ between them. It follows that a \emph{green state} situated at the end of a \emph{branch} is associated with one pre-image of $a$ while a \emph{green state} lying on a \emph{branch} containing two \emph{states} of different \emph{colors} on its two sides is associated with two pre-images of $a$. Now, the fixed point $a$ is itself a \emph{pre-image} of $a$ which separates two \emph{black states} lying on different \emph{branches} as the \emph{branches} have been \emph{canonically ordered} (See Theorem \ref{black:loop:length:3}). Thus, if $g$ is the number of \emph{green} states, we must have $m \geqslant 2 (g - 3) + 3$ which yields $g \leqslant \frac{m}{2} + \frac{3}{2} <  \frac{m}{2}+2$. Thus, the cardinality of  $\mathcal{G}(P)$ is strictly less than $\frac{m}{2}+2$. Similarly, the cardinality of  $\mathcal{R}(P)$ is strictly less than $\frac{m}{2}+2$. 
\end{proof}

Our objective is to establish that the quantity $\chi(P)$ admits an upper bound depending solely on the \emph{modality} $m$ of the cycle $P$. In view of Theorem~\ref{bifurcation:one:third}, \emph{triod--twist cycles} undergo a qualitative \emph{bifurcation} at the critical \emph{rotation number} $\rho=\tfrac{1}{3}$. This observation naturally motivates a separate analysis of the bound on $\chi(P)$ in the three regimes $\rho=\tfrac{1}{3}$, $\rho<\tfrac{1}{3}$ and $\rho>\tfrac{1}{3}$. In the case $\rho=\tfrac{1}{3}$, the cycle $P$ is necessarily the \emph{primitive} $3$--cycle and hence $\chi(P)$ is trivially bounded above by $1$. We therefore turn our attention to the remaining two cases, which will be examined individually.

\subsection{Case I}

We begin with the case $\rho < \tfrac{1}{3}$.  
According to Theorem~\ref{bifurcation:one:third}, in this situation, the cycle $P$ consists entirely of \emph{green} and \emph{black} points. Also, recall from Definition \ref{non:decreasing} and Theorem \ref{tri-od:rot:twist:order:inv} that in this case for any two points $x,y \in P$ in same \emph{branch} of $Y$, with $x >  y$, we have, $L(y) >  L(x)$. 

\begin{definition}
	A finite sequence $\mathcal{S} = \{x_i : i = 0, 1, 2, \dots, n\}$ of points in $P$ is called an \emph{$f$-train} of length $n$ if it satisfies $x_{i+1} \geqslant f(x_i)$ for every $i \in \{0, 1, 2, \dots, n\}$.  
	An $f$-\emph{train} $\mathcal{S}$ is called \emph{black} if all its points are \emph{black}, \emph{green} if all its points are \emph{green}, and \emph{mixed} if both \emph{colors} appear in $\mathcal{S}$.
\end{definition}

Observe that if $x,y \in P$ and $\mathcal{S} = \{x_0, x_1, x_2, \dots x_n\}$ is an $f$-\emph{train}  of length $n$ with $x_0 = x$ and $x_n =y$,  then $L(y) - L(x) = L(x_n) - L(x_0) \leqslant  \displaystyle \sum_{i=0}^{n-1} L(f(x_i)) - L(x_i)$.

\begin{figure}[H]
	\centering
	\includegraphics[width=0.4\textwidth]{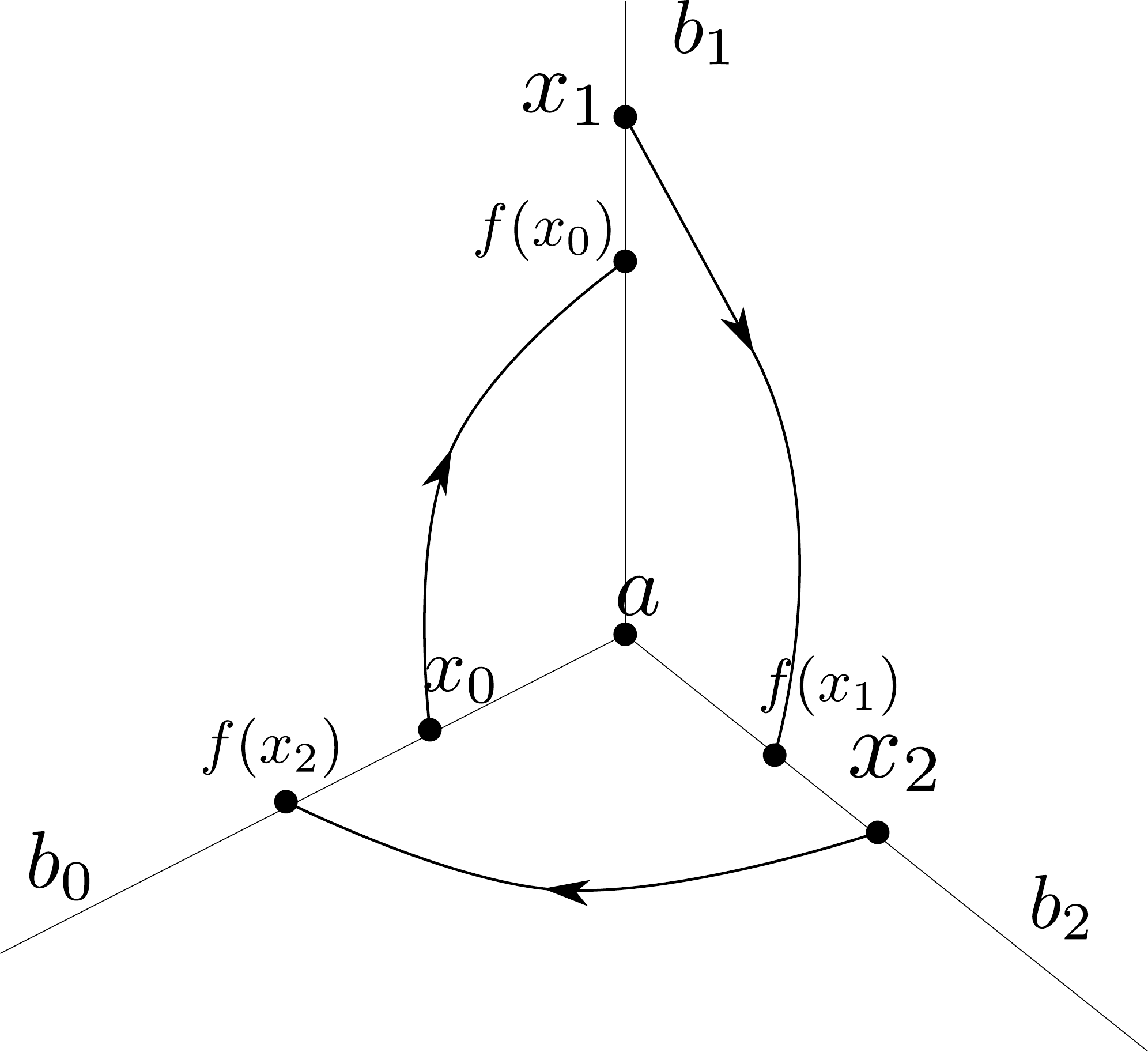}
	\caption{}
	\label{drawing r4}
\end{figure}

\begin{lemma}\label{black:train}
	Let $\mathcal{S} = \{x_i : i = 0, 1, 2, \dots, n\}$ be a black train such that $x_0$ and $x_n$ lie on the same branch.  
	Then $x_n \geqslant x_0$.
\end{lemma}

\begin{proof}
	Since $x_0$ and $x_n$ lie on the same \emph{branch}, $n$ is some multiple of $3$.  
	Observe that the points $x_0$, $f(x_2)$, and $x_3$ belong to the same branch.  
	We claim that $f(x_2) \geqslant x_0$.  
	Assume, to the contrary, that $x_0 >  f(x_2)$.  
	By Theorem~\ref{necesary:condition:triod:twist}, the collection of points lying in the intervals $[a, x_0]$, $[a, x_1]$, and $[a, x_2]$ would then be invariant under $f$, which is a contradiction. 
	Thus, $x_3 \geqslant f(x_2) \geqslant x_0$ (see Figure~\ref{drawing r4}).  
	Applying the same reasoning iteratively gives $ x_6 \geqslant f(x_5) \geqslant x_3$, $x_9 \geqslant f(x_8) \geqslant x_6$, and so on.  
	Hence, the result follows.
\end{proof}

\begin{theorem}\label{green: state:phase:bound}
	For every $A \in \mathcal{G}(P)$, one has $\chi(A) \leqslant 1 - 3\rho$.
\end{theorem}

\begin{proof}
	Let $t(A)$ be the \emph{black point} of $P$ adjacent to $i(A)$ towards the \emph{branching point} $a$. 	From Theorems~\ref{bifurcation:one:third} and~\ref{black:loop:length:3}, together with the fact that the sets  
	$\{p \in P : t(A) \geqslant p\}$ and $\{p \in P : p \geqslant e(A)\}$  
	are not invariant under $f$, it follows that there exist \emph{black} points $x, y, z \in P$ such that  
	$t(A) \geqslant x$, $f(x) \geqslant y$, $f(y) \geqslant z$, and $f(z) \geqslant e(A)$ (see Figure~\ref{drawing r5}).  By Definitions~\ref{code:function} and \ref{non:decreasing}, we obtain  
	$\chi(A) \leqslant L(x) - L(f(z)) = L(f(x)) - \rho + \tfrac{1}{3} - L(z) - \rho + \tfrac{1}{3} \leqslant L(y) - L(z) - 2\rho + \tfrac{2}{3} = L(f(y)) - 3\rho + 1 - L(z) \leqslant 1 - 3\rho$.  
	Hence, the claim follows.
\end{proof}

\begin{figure}[H]
	\centering
	\includegraphics[width=0.5\textwidth]{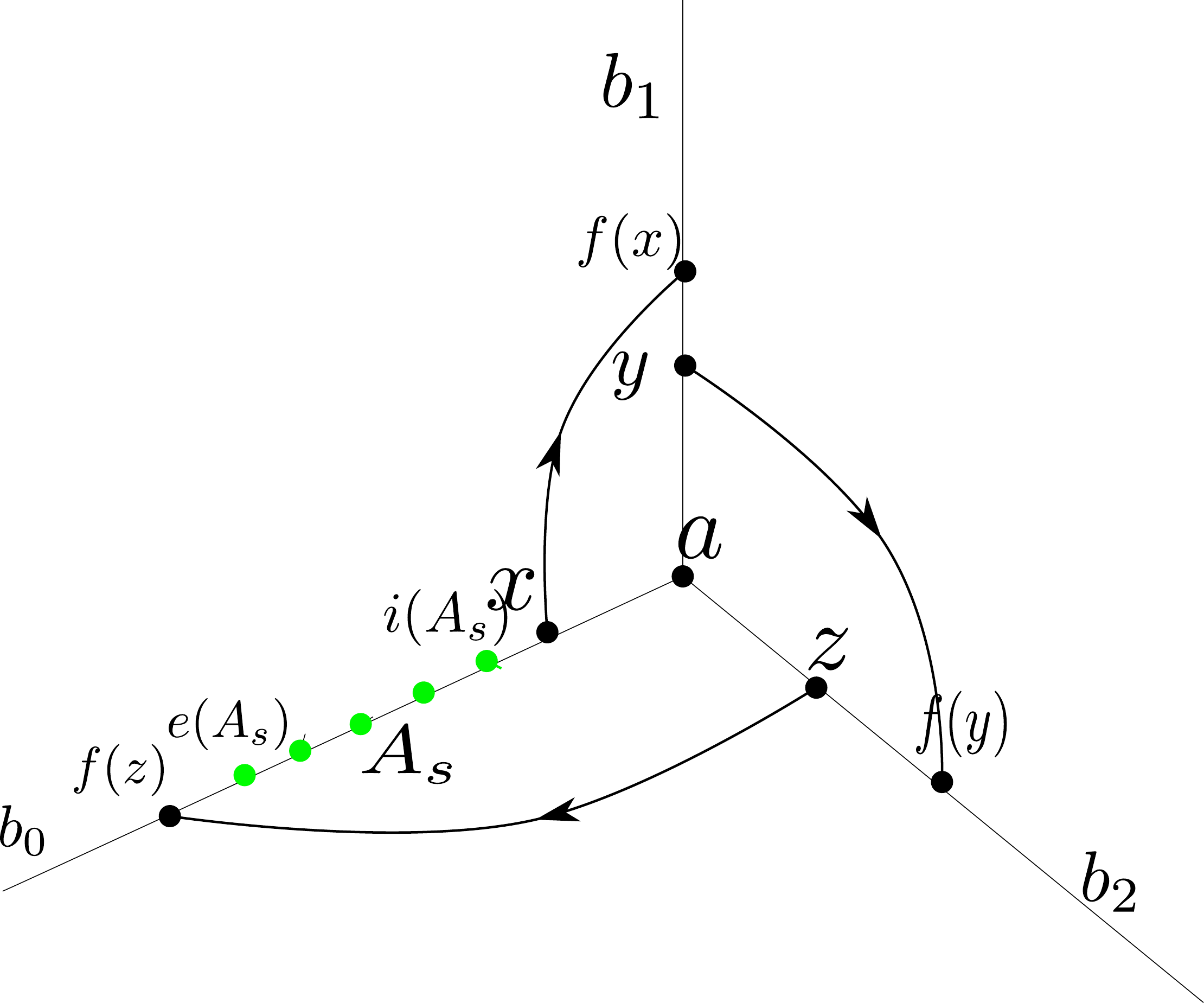}
	\caption{}
	\label{drawing r5}
\end{figure}

Two \emph{green} states $A$ and $B$ with $A >  B$ are called \emph{adjacent} if there exists no \emph{green state} $C$ satisfying $A >  C >  B$.

\begin{definition}\label{green:country}
	Two \emph{adjacent green states} $A, B \in \mathcal{G}(P), A >  B$ are said to belong to the same \emph{country} if there exist points $a \in A$ and $b \in B$ such that $b \geqslant  f(a)$, that is, there exists a  \emph{green train} $\mathcal{S} = \{a, b\}$ of length $1$ that connects $a$ and $b$.

\end{definition}

Let $\mathcal{C}(P)$ be the collection of all \emph{green countries} of $P$.

\begin{theorem}\label{green:country:phase:bound:1}
	Let $C \in \mathcal{C}(P)$ be a green-country composed of $m$ green states. Then, the following inequality holds:
	$\chi(C) \leqslant m(1 - 2\rho) - \rho$. 
\end{theorem}

\begin{proof}
	Let $A$ and $B$ be two \emph{consecutive green states} in $C$ such that $(A > B)$.  
	According to Definition~\ref{green:country}, there exist points $a \in A$ and $b \in B$ connected by a \emph{green train} $\mathcal{S} = \{a, b\}$ of length $1$.  
	Hence, we have $L(b) - L(a) \leqslant \rho$.
	
	Assume that $C$ contains $m$ \emph{green states}, and let $D$ and $E$ denote, respectively, the \emph{green states} in $C$ that are nearest to and farthest from the \emph{branching point} $a$. The maximal variation of the \emph{code function} between a point $d\in D$ and a point $e\in E$ is attained by moving through each \emph{green state} once and making  $m-1$ times \emph{inter-state} jumps. Thus, by   Definition~\ref{non:decreasing} and  Theorem~\ref{green: state:phase:bound},  for any $d \in D$ and $e \in E$, it follows that $L(d) - L(e) \leqslant (m - 1)\rho + m(1 - 3\rho) = m(1 - 2\rho) - \rho$. 
	This establishes the desired bound.
\end{proof}

\begin{definition}\label{closest:country}
	Let $\Phi_{i} : \mathcal{C}(P) \to \mathcal{C}(P)$ for $i  \in \{1,2\}$ denote the  two mappings defined as follows:  
	\begin{enumerate}
		\item Take $C \in \mathcal{C}(P)$. Let $C \in b_{j}$ for some $j \in \{0,1,2\}$. Choose the \emph{green state} $A$ of $C$ closest to the \emph{branching point} $a$.  Clearly,  $f(i(A)) = c(A) \in b_j$ is a \emph{black point}, and its image $f(c(A)) \in b_{j+1}$.  
		If $f(c(A))$ belongs to a \emph{green country} $D \in b_{j+1}$, then we assign $\Phi_{1}(C) = D$.  	If $f(c(A))$ is \emph{black}, but there exists a \emph{green country} $E \in b_{j+1}$ such that $E >  f(c(A))$, we set $\Phi_{1}(C) = E$.  	On the other hand, if $f(c(A))  \in b_{j+1}$ is \emph{black} and there is no \emph{green country} $E \in b_{j+1}$ satisfying $E > f(c(A))$, we say that $\Phi_{1}(C)$ \emph{does not exist} (see Figure~\ref{drawing r6}).
		
		\item Suppose $f(c(A)) \in b_{j+1}$ is \emph{black} and $f^{2}(c(A)) \in b_{j+2}$ lies in a \emph{green country} $F \in b_{j+2}$, we set $\Phi_{2}(C) = F$. If $f^{2}(c(A))  \in b_{j+2}$ is \emph{black}, but there exists a \emph{green country} $G \in b_{j+2}$ with $G >  f^{2}(c(A))$, we define $\Phi_{2}(C) = G$.  
		Now, if $f^{2}(c(A))  \in b_{j+2}$ is \emph{black} and no \emph{green country} $G \in b_{j+2}$ satisfies $G > f^{2}(c(A))$, we say that $\Phi_{2}(C)$ \emph{does not exist} (see Figure~\ref{drawing r7}).
	\end{enumerate}
\end{definition}

\begin{figure}[H]
	\centering
	\includegraphics[width=0.5\textwidth]{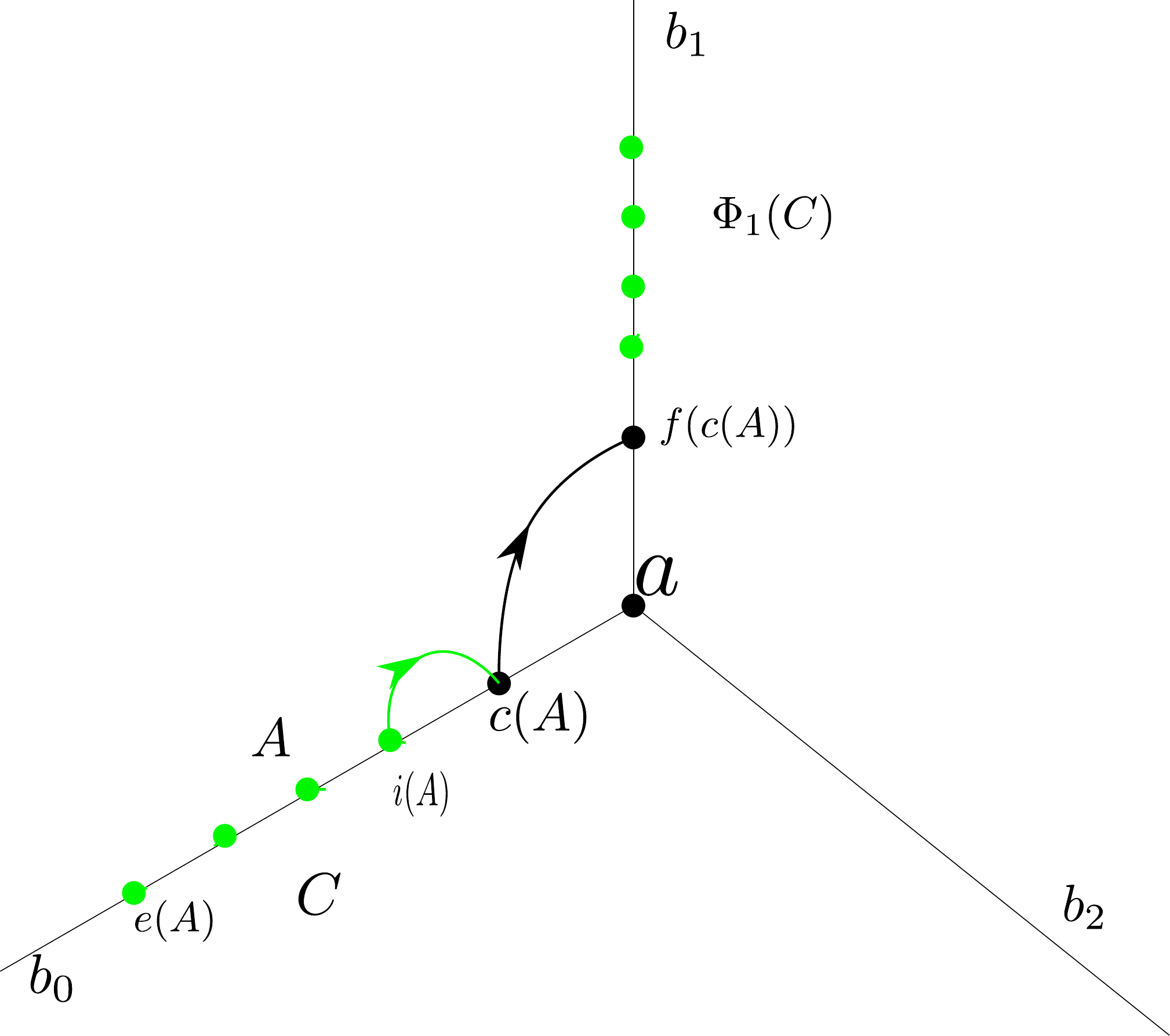}
	\caption{}
	\label{drawing r6}
\end{figure}

\begin{theorem}\label{no:phi:1:2}
	Assume that for some green country $C$ located within a branch $b_0$ of $Y$, both $\Phi_1(C)$ and $\Phi_2(C)$ do not exist. Then, neither branch $b_1$ nor  the branch $b_2$ of $Y$ contains any green point, and all green points in $b_0$ together constitute a single country.
\end{theorem}

\begin{proof}
	Let $A$ denote the \emph{green state} of $C$ that lies closest to the \emph{branching point} $a$.  Then, the point $c(A) = f(i(A)) \in b_0$ is \emph{black}.  
	Since $\Phi_1(C)$ and $\Phi_2(C)$ do not exist, it follows that the points $f(c(A)) \in b_1$ and $f^2(c(A)) \in b_2$ are both \emph{black}. Moreover, there are no \emph{green points} $t \in b_1$ with $t >  f(c(A))$, and no \emph{green  points} $s \in b_2$ with $s >  f^2(c(A))$.  
	
	By Lemma~\ref{black:train}, Theorem~\ref{necesary:condition:triod:twist},  combined with the non-existence of $\Phi_1(C)$ and $\Phi_2(C)$ and the fact that $A$ is the \emph{green state} of $C$ nearest to the \emph{branching point} $a$, we conclude that the points $c(A)$, $f(c(A))$, and $f^2(c(A))$ are precisely those elements of $P$ that are closest to $a$ on their respective \emph{branches}. Hence, the claim follows.
\end{proof}

\begin{theorem}\label{phi:cube:closer}
	Let there be a green country $C$ contained within a branch $b_0$ of $Y$ such that $\Phi_1(C)$, $\Phi_1^2(C)$, and $\Phi_1^3(C)$ all exist. Then $C >  \Phi_1^3(C)$, unless no green country $C'$ exists with $C > C'$.
\end{theorem}

\begin{proof}
	Suppose, for the sake of contradiction, that there exists a \emph{green country} $C'$ with $C >  C'$ and $\Phi_1^3(C) >  C$.  	Let $A_0, A_1$ and $A_2$ denote the \emph{green states} of $C$, $\Phi_1(C)$ and $\Phi_1^2(C)$ respectively which are closest to the \emph{branching point} $a$ in the corresponding \emph{branches} $b_0, b_1$, and $b_2$.  For each $i \in \{0,1,2\}$, define $c(A_i) = f(i(A_i))$.  	Let $\eta_{i+1}$ denote the point that is closest to $a$ among $f(c(A_{i}))$ and $c(A_{i+1})$ for $i \in \{0,1\}$ and $\eta_0$ be the point closest to $a$ among $f(c(A_2))$ and  $c(A_0)$. 
	
	By Lemma~\ref{black:train} and Theorems  \ref{necesary:condition:triod:twist} and \ref{black:loop:length:3}, the set of points of $P$ lying farther away from $a$ than $\eta_i, i \in \{0, 1,2\}$  remains invariant under the map $f$. This contradicts the assumption that $C >  C'$, and hence the claim follows.
\end{proof}

\begin{figure}[H]
	\centering
	\includegraphics[width=0.5\textwidth]{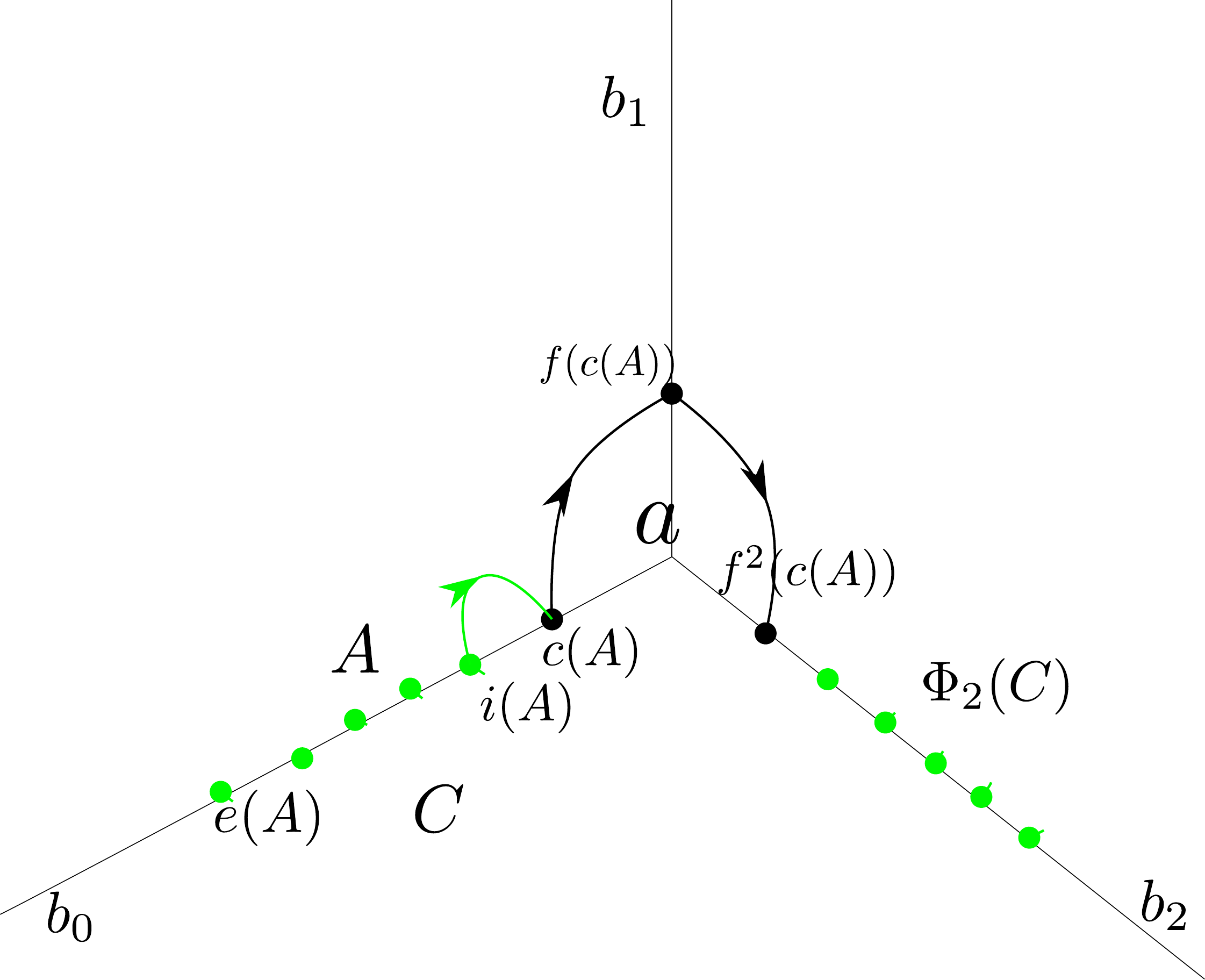}
	\caption{}
	\label{drawing r7}
\end{figure}

\begin{theorem}\label{green:black:movement}
	Let $x \in P$ be any black point. Then there exists a green point $z \in P$ such that $L(x) - L(z) \leqslant \rho$.
\end{theorem}

\begin{proof}
	Assume that the \emph{branches} of $Y$ are \emph{canonically ordered} so that, in each \emph{branch} $b_i$ ($i \in \{0,1,2\}$), the point $c_i$ nearest to the \emph{branching point} $a$ is \emph{black}.  
	
	We begin by showing that at least one of the points $c_i$ must be the image of a \emph{green point}. Suppose, for contradiction, that each $c_i$ is the image of a \emph{black point}. Let $d_i$ and $e_i$ denote, respectively, the pre-image and image of $c_i$. Then the order relations satisfy $d_0 \geqslant c_2$, $d_2 \geqslant c_1$, and $d_1 \geqslant c_0$. Since $\rho \neq \frac{1}{3}$, at least one of these inequalities is strict.  
	
	Assume $d_1 > c_0$. Because $P$ is \emph{order-preserving}, we have $f(d_1) > f(c_0)$, which implies $c_1 > e_0$. This contradicts the definition of $c_1$ as the point in \emph{branch} $b_1$ closest to $a$. Therefore, our assumption is false, and the claim follows that at least one $c_i$ is the image of a \emph{green point}.  
	
	Without loss of generality, let $d_0$ be such a \emph{green} point. Clearly, $L(c_0) - L(d_0) = \rho$. Hence, for any \emph{black} point $\mu \in b_0$, we have $L(\mu) - L(d_0) \leqslant L(c_0) - L(d_0) = \rho$.  
	
			\smallskip 
		\smallskip 
		\smallskip
		\smallskip 
		\smallskip

	Now, consider a \emph{black} point $\eta \in b_1$. Two cases arise:  
	
	Case 1: If $c_1 = f(c_0)$, then
	$L(\eta) - L(d_0) \leqslant L(c_1) - L(d_0) = L(f(c_0)) - L(d_0) = L(c_0) + \rho - \tfrac{1}{3} - L(d_0) = 2\rho - \tfrac{1}{3} <  \rho.
	$
	
	\smallskip 
	\smallskip

	Case 2: If $f(c_0) \neq c_1$, then $c_1$ must be the image of a \emph{green point} $z_1 \in b_1$. Indeed, otherwise,  by Theorem \ref{necesary:condition:triod:twist}, $f^{-1}(c_1)$ would lie in \emph{branch} $b_0$ with $c_0 > f^{-1}(c_1)$. This would contradict the fact that $c_0$ is the point of $P$ closest to the \emph{branching point} $a$ in the \emph{branch} $b_0$. Now, $L(c_1) - L(z_1) = \rho$, and therefore $L(\eta) - L(z_1) \leqslant L(c_1) - L(z_1) = \rho$.  
	
	\smallskip 
	\smallskip 
		\smallskip
		\smallskip 
		\smallskip

	Finally, consider a \emph{black} point $\xi \in b_2$. From the previous step, there exists a \emph{green} point $g$ such that $L(c_1) - L(g) \leqslant \rho$.  
	
	Case 1: If $f(c_1) = c_2$, then $L(\xi) - L(g) \leqslant L(c_2) - L(g) = L(c_1) + \rho -\frac{1}{3} - L(g)\leqslant 2\rho- \frac{1}{3} \leqslant \rho$.  
	
	\smallskip 
	\smallskip

	Case 2: If $f(c_1) \neq c_2$, then, as before, $c_2$ must be the image of a \emph{green point} $z_2 \in b_2$. Otherwise, by Theorem \ref{necesary:condition:triod:twist},  $f^{-1}(c_2)$ would lie in branch $b_1$, with $c_1 > f^{-1}(c_2)$, this would contradict the definition of $c_1$ as the closest point to $a$ point of $P$ in $b_1$. Thus, $L(c_2) - L(z_2) = \rho$, and consequently $L(\xi) - L(z_2) \leqslant L(c_2) - L(z_2) = \rho$.  
	
	This completes the proof.
\end{proof}

\begin{figure}[H]
	\centering
	\includegraphics[width=0.4\textwidth]{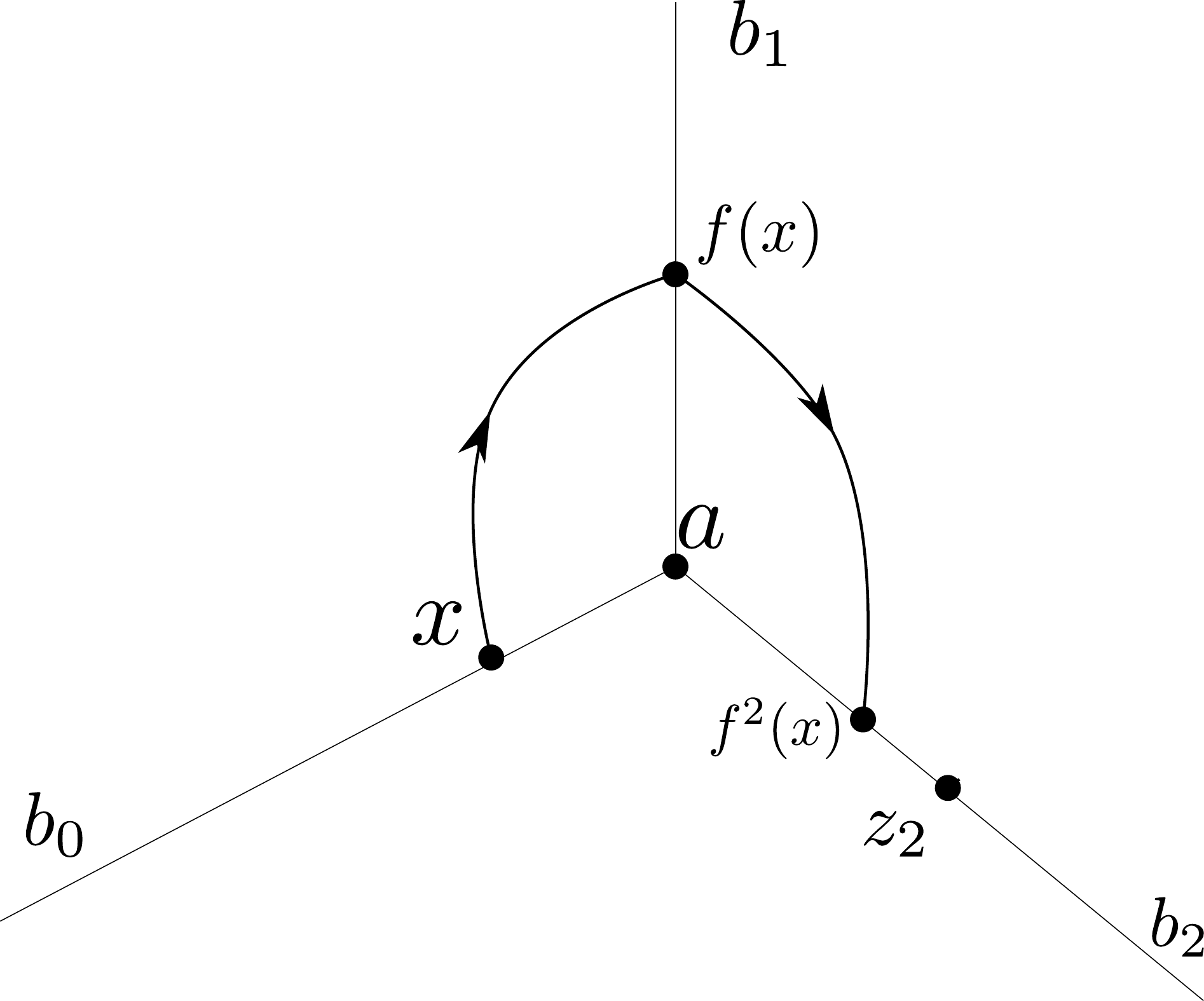}
	\caption{}
	\label{drawing r8}
\end{figure}

\begin{theorem}\label{black:green:movement}
	Let $x \in P$ be any black point. Then there exists a green point $z \in P$ such that $L(z) - L(x) \leqslant \rho$.
\end{theorem}

\begin{proof}
	Assume that $x \in b_0$.  	If there exists a \emph{green point} $z_0 \in b_0$ satisfying $z_0 >  x$, then clearly 	$L(z_0) - L(x) < 0$.  	Next, if there exists a \emph{green point} $z_1 \in b_1$ such that $z_1 \geqslant f(x)$, then  $L(z_1) - L(x) = L(z_1) - \big\{ L(f(x)) - \rho + \tfrac{1}{3} \big\} = L(z_1) - L(f(x)) + \rho - \tfrac{1}{3} \leqslant \rho - \tfrac{1}{3} < \rho$.
	
	Suppose now that there are no \emph{green points} in either of the sets $\{ t \in P : t >  x \}$ or $\{ t \in P : t \geqslant f(x) \}$.  
	Since $f(x)$ is \emph{black}, it follows that $f^2(x) \in b_2$.  
	If there exists a \emph{green point} $z_2 \geqslant f^2(x)$ (see Figure~\ref{drawing r8}), then  $L(z_2) - L(x) = L(z_2) - \big\{ L(f(x)) - \rho + \tfrac{1}{3} \big\}  = L(z_2) - \big\{ L(f^2(x)) - \rho + \tfrac{1}{3} \big\} + \rho - \tfrac{1}{3}  = L(z_2) + 2\rho - \tfrac{2}{3} - L(f^2(x)) \leqslant 2\rho - \tfrac{2}{3} \leqslant \rho,$ as $\rho < \frac{1}{3}$. 

	Finally, suppose that no \emph{green point} $z_2$ exists with $z_2 \geqslant f^2(x)$.  
	Then, by Lemma~\ref{black:train}, all points of $P$ farther away than $a$ than $x$, $f(x)$, and $f^2(x)$ from the \emph{branching point} $a$ are \emph{black}.  This means either the points of $P$ farther away from $a$ than $x, f(x) , f^2(x)$ are invariant under $f$ or   $P$  contains no \emph{green points}, and its \emph{rotation number} is $\tfrac{1}{3}$, both of which leads to a contradiction. 	Among the quantities $\rho$, $\rho - \tfrac{1}{3}$, and $2\rho - \tfrac{2}{3}$, the largest is $\rho$, since $\rho \leqslant \frac{1}{3}$. 	Hence, the conclusion follows.
\end{proof}

\begin{theorem}\label{thm:chi-bound-final}
	Let $P$ be a triod–twist cycle of modality $m$ and rotation number $\rho$, $0 \leqslant \rho < \tfrac{1}{3}$. Then, 	$\chi(P) <  m + 3$. 
\end{theorem}

\begin{proof}
	By Theorems~\ref{green:black:movement} and~\ref{black:green:movement}, 
	every \emph{black point} of $P$ lies within $\rho$ (in \emph{code}) of some \emph{green point}, 
	both above and below.  Consequently, the total 
	oscillation $\chi(P)$ of the \emph{code function} satisfies $	\chi(P) \leqslant  \chi(G(P)) + \rho$, where $G(P)$ denotes the collection of  all \emph{green points} of $P$. 
	Hence, it suffices to obtain an upper bound for $\chi(G(P))$.
	
	Let $\mathcal{C}(P) = \{C_1, C_2, \dots, C_{\ell}\}$ denote the collection 
	of all \emph{green countries} of $P$, and let $k_i$ be the number of 
	\emph{green states} contained in the \emph{green country} $C_i$. By Theorem~\ref{green:country:phase:bound:1},
	each \emph{green country} satisfies $\chi(C_i) \leqslant  k_i(1 - 2\rho) - \rho$.	Since $P$ is finite, the set of \emph{green countries} is also finite. 
	By Definition~\ref{closest:country} and Theorems~\ref{no:phi:1:2}–\ref{phi:cube:closer},
	for any two \emph{green points} $g_1, g_2 \in G(P)$ there exists a 
	\emph{mixed train} connecting them that visits each \emph{green country} 
	at most once. Every time the \emph{train} moves from one \emph{green country} to another, 
	the \emph{code} increases by less than $2\rho - \tfrac{1}{3}$. 
	Since at most $\ell - 1$ such inter-\emph{country} transitions can occur, we have, 	$	\chi(G(P))	\leqslant  \displaystyle  \sum_{i=1}^{\ell} \big(k_i(1 - 2\rho) - \rho\big)
	+ (\ell - 1)\!\left(2\rho - \frac{1}{3}\right)$.
	
	Let $K = \displaystyle  \sum_{i=1}^{\ell} k_i$ denote the total number of \emph{green states}. 
	Then, $\chi(G(P)) \leqslant K(1 - 2\rho) - \ell\rho + (\ell - 1)\!\left(2\rho - \frac{1}{3}\right) = K - (2K - \ell + 2)\rho - \frac{\ell - 1}{3}$.

	Since $K \geqslant  \ell$, the coefficient of $\rho$ is positive, 
	so the right-hand side is maximized when $\rho \to 0$. 
	Thus, $\chi(G(P)) \leqslant  K - \frac{\ell - 1}{3}$.  Substituting this bound into the earlier inequality gives $	\chi(P) \leqslant K - \frac{\ell - 1}{3} + \rho \leqslant  K - \frac{\ell - 2}{3}$. By Lemma \ref{first:lemma}, $K <   \frac{m}{2} + 2$ and $\ell > 1$, we have  $-\tfrac{\ell - 2}{3} <  1$. Hence, $\chi(P) \leqslant K + 1 <  (\frac{m}{2}) + 3 < m + 3$. Therefore, the claimed upper bound holds for all \emph{triod–twist cycles} 	with $0 \leqslant \rho \leqslant \tfrac{1}{3}$. 
\end{proof}

\subsection*{3.2. Case II: $\rho > \tfrac{1}{3}$}

According to Theorem \ref{bifurcation:one:third}, when $\rho >  \tfrac{1}{3}$, every \emph{triod–twist cycle} $P$ consists exclusively of \emph{red} and \emph{black} points.
By Definition \ref{non:decreasing} and Theorem \ref{tri-od:rot:twist:order:inv}, if $x,y\in P$ lie on the same \emph{branch} of $Y$ with $x > y$, then $L(x) >  L(y)$.

\begin{theorem}\label{rho:less:half}
	Every triod--twist pattern $\pi$ has rotation number $\rho < \tfrac12$.
\end{theorem}

\begin{proof}
	Let $P$ be a cycle that \emph{exhibits} the pattern $\pi$, and let $f$ be the associated $P$--\emph{linear} map. Denote by $b$ and $r$ the numbers of \emph{black} and \emph{red} points of $P$, respectively. We first claim that $b>r$.
	
	Assume, to the contrary, that $r \geqslant b$. Since each point of the cycle must eventually return to the \emph{branch} on which it lies, the total numbers of \emph{black} and \emph{red} points must each be divisible by three. Hence, there exist $k_1,k_2 \in \mathbb{N}$ such that	$b = 3k_1 \quad \text{and} \quad r = 3k_2$, with $k_2 \geqslant k_1$. By the definition of the \emph{rotation number} (see Section~1), we obtain $\rho = \frac{b+2r}{3(b+r)} = \frac{3(k_1+2k_2)}{9(k_1+k_2)}, $ $ k_2 \geqslant k_1 $. Clearly, the \emph{rotation pair} of $P$ is not \emph{co-prime} and hence by Theorem \ref{result:1}, $P$ cannot be a \emph{triod-twist} cycle. 
	Hence, the assumption $r \geqslant b$ is false, and we conclude that $b>r$.	Substituting this inequality into the expression for $\rho$, we obtain	$\rho = \frac{b+2r}{3(b+r)} < \frac{b+2b}{3(2b)} = \frac12$, which completes the proof.
\end{proof}

\begin{theorem}\label{red:state:phase:bound}
	For every red state $R \in \mathcal{R}(P)$ of a triod–twist cycle $P$ with rotation number $\rho >  \tfrac{1}{3}$, one has	$\chi(R) \leqslant  3\rho - 1$,where $\chi(R) = \sup\{\,L(x) - L(y) : x, y \in R\,\}$ denotes the oscillation of the code function $L$ on $R$.
\end{theorem}

\begin{proof}
	Let $R$ be a \emph{red state} of $P$, and let $b_0$ denote the \emph{branch} of $Y$ containing $R$.  Let $i(R)$ be the point of $R$ closest to $a$  and $\xi(R)$ be the \emph{black point} just adjacent to $i(R)$ in the direction towards $a$. 	Such a point exists by the \emph{canonical ordering} of \emph{branches} (Theorem \ref{black:loop:length:3}).
	
	Because the sets 	$\{t \in P : i(R) \geqslant t\}$ and 	$\{t \in P : t \geqslant e(R)\}$ 
	are not invariant under $f$, at least one of the following two cases must occur.

	\smallskip 
\smallskip 
\smallskip 
\smallskip

	Case 1: There exists a \emph{red point} $\eta_1(R)$ such that  
	$f(\xi(R)) \geqslant \eta_1(R)$ and $f(\eta_1(R)) \geqslant e(R)$ (see Figure~\ref{drawing r9}). We have, 	$\chi(R) \leqslant L(e(R)) - L(i(R)) \leqslant L(f(\eta_1(R))) - L(\xi(R)) = L(\eta_1(R)) - L(f(\xi(R))) + 2\rho - \tfrac{1}{3} - \tfrac{2}{3} = L(\eta_1(R)) - L(f(\xi(R))) + 2\rho - 1 \leqslant 2\rho - 1$. Thus, $\chi(R) \leqslant 2\rho - 1$ in this case.

		\smallskip 
	\smallskip 
	\smallskip 
	\smallskip 

	Case 2: There exist \emph{black points} $\eta_2(R), \eta_3(R) \in P$ such that  
	$f(\xi(R)) \geqslant \eta_2(R)$, $f(\eta_2(R)) \geqslant \eta_3(R)$, and $f(\eta_3(R)) \geqslant e(R)$.	Then
	$\chi(R) \leqslant L(e(R)) - L(i(R)) \leqslant L(f(\eta_3(R))) - L(\xi(R)) =  L(\eta_3(R)) - L(f(\xi(R))) + 2\rho - \tfrac{2}{3} \leqslant L(\eta_3(R)) - L(\eta_2(R)) + 2\rho - \tfrac{2}{3} = L(\eta_3(R)) - L(f(\eta_2(R))) + 3\rho - 1 \leqslant 3\rho - 1$.
	
		\smallskip 
	\smallskip 
	\smallskip 
	\smallskip

	In both cases, the difference in \emph{code values} between the endpoints of $R$ is bounded above by $3\rho - 1$.  Hence,	$\chi(R) \leqslant  3\rho - 1$. 	This completes the proof.
\end{proof}

\begin{figure}[H]
	\centering
	\includegraphics[width=0.5\textwidth]{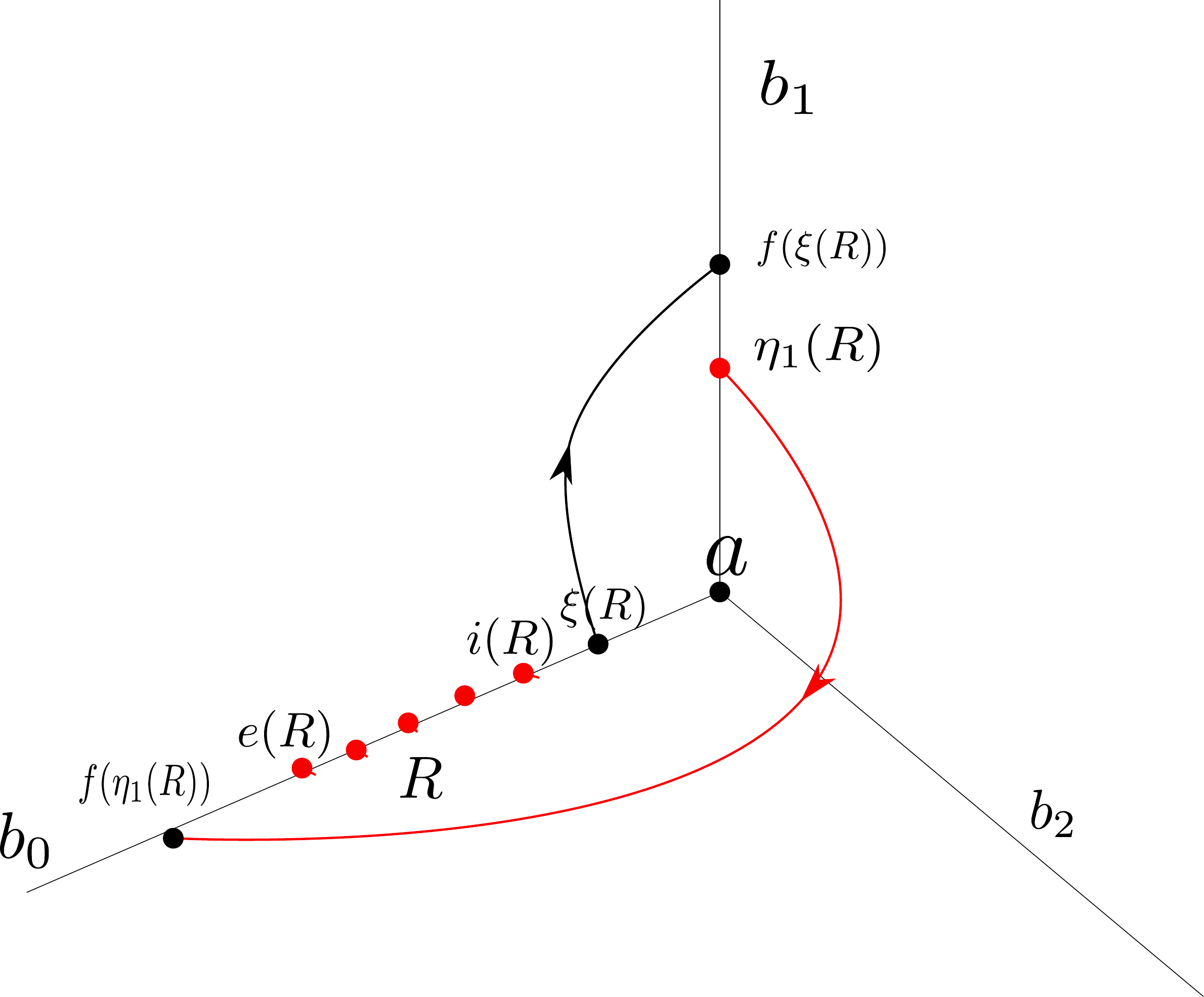}
	\caption{}
	\label{drawing r9}
\end{figure}

\begin{theorem}\label{movement:b0:b2}
	Suppose the branches $b_0$ and $b_2$ of the triod $Y$ both contain red states of the triod–twist cycle $P$.  	Let $R_0$ and $R_2$ be the red states in $b_0$ and $b_2$, respectively that lie closest to the branching point $a$.  Then there exist points $x \in R_0$ and $y \in R_2$ such that $L(x) - L(y) < 2\rho - \tfrac{2}{3}$, where $\rho$ is the rotation number of $P$.
\end{theorem}

\begin{proof}
	Since $\rho >  \tfrac{1}{3}$, the cycle $P$ consists only of \emph{red} and \emph{black points} (by Theorem \ref{bifurcation:one:third}), and the \emph{branches} are \emph{canonically ordered} (Theorem \ref{black:loop:length:3}).  We consider the relative positions of $f(e(R_0))$ and $i(R_2)$ where $e(R_0)$ and $i(R_2)$ denote respectively, the \emph{outer endpoint} (\emph{point farthest from} $a$) of $R_0$ and the \emph{inner endpoint} (\emph{point closest to} $a$) of $R_2$.
	
	\smallskip 
\smallskip 
\smallskip 
\smallskip 

	Case 1: $f(e(R_0)) \geqslant i(R_2)$.  Then, by the recurrence rule for the \emph{code function}, $L(i(R_2)) - L(e(R_0)) = L(i(R_2)) -\big(L(f(e(R_0))) - \rho + \tfrac{2}{3}\big) 	= L(i(R_2)) - L(f(e(R_0))) + \rho - \tfrac{2}{3} \leqslant \rho - \tfrac{2}{3}$.	
	
		\smallskip 
	\smallskip 
	\smallskip 
	\smallskip 

	Case 2: $i(R_2) > f(e(R_0))$.  	Since, the sets $\{t \in P : i(R_0) \geqslant t\}$ and 
	$\{t \in P : t \geqslant i(R_2)\}$ are not invariant under $f$.  Hence, there exist \emph{black points} $\xi, \eta \in P$ such that  $i(R_0) > \xi, 	f(\xi) \geqslant \eta$ and $f(\eta) \geqslant i(R_2)$.  Using the same recurrence relation and the inequalities, we obtain $L(i(R_2)) - L(i(R_0)) \leqslant L(f(\eta)) - L(i(R_0)) =L(\eta) + \rho - \frac{1}{3} - L(i(R_0)) \leqslant L(f(\xi)) + \rho - \frac{1}{3} - L(i(R_0)) = L(\xi) + 2 \rho - \frac{2}{3} - L(i(R_0)) = (2 \rho - \frac{2}{3}) + (L(\xi) - L(i(R_0))) < 2 \rho - \frac{2}{3}$

	\smallskip 
\smallskip 
\smallskip 
\smallskip

	In both cases, the difference of \emph{code values} between the \emph{red states} $R_0$ and $R_2$ is bounded above by $2\rho - \tfrac{2}{3}$.  Therefore, there exist $x \in R_0$ and $y \in R_2$ such that	$L(x) - L(y) < 2\rho - \tfrac{2}{3}$. This completes the proof.
\end{proof}

\begin{theorem}\label{red:adjacent:states}
	Suppose the branch $b_0$ of the triod $Y$ contains two adjacent red states $R$ and $S$ with $R > S$.  Then for every $r \in R$ and $s \in S$, we have,  $	L(r) - L(s) \leqslant 3\rho - 1$, where $\rho$ is the rotation number of the triod–twist cycle $P$.
\end{theorem}

\begin{proof}
	Move along the \emph{branch} $b_0$ from the \emph{red state} $S$ toward the \emph{branching point} $a$.  	Let $\xi_1$ denote the \emph{black point} of $P$ lying closest to $i(S)$ (the \emph{inner endpoint} of $S$), towards the \emph{branching point} $a$.  Such a point $\xi_1$ exists because, by Theorem~\ref{black:loop:length:3}, the points of $P$ nearest to the \emph{branch point} on each \emph{branch} are \emph{black}.  	Since the sets 
	$\{t \in P : i(S) \geqslant t\}$ and 	$\{t \in P : t \geqslant e(R)\}$ 	are not invariant under $f$, one of the following two situations must occur.
	
	\smallskip 
\smallskip 
\smallskip 
\smallskip 

Case 1: Suppose there exists a \emph{red point} $\eta_1$ for which  
$f(\xi_1) \geqslant \eta_1$ and $f(\eta_1) \geqslant e(R)$.  In this situation, for any $r \in R$ and $ s \in S$, we obtain $L(r) - L(s) \leq L(e(R)) - L(i(S))  	\leq L(f(\eta_1)) - L(\xi_1) = \{ L(\eta_1) + \rho - \frac{2}{3}\} - \{ L(f(\xi_1)) - \rho + \frac{1}{3}\} = \bigl(L(\eta_1) - L(f(\xi_1))\bigr) + 2\rho - 1 \leqslant 2\rho - 1 $.

		\smallskip 
	\smallskip 
	\smallskip 
	\smallskip 
	
	Case 2: There exist \emph{black points} $\xi_2$ and  $\xi_3 \in P$ such that  
	$f(\xi_1) \geqslant \xi_2$, $f(\xi_2) \geqslant \xi_3$, and $f(\xi_3) \geqslant e(R)$.  	Then, for any $r \in R$ and $ s \in S$, we obtain,  $L(r) - L(s) \leqslant L(e(R)) - L(i(S))  \leqslant L(f(\xi_3)) - L(\xi_1)  = L\bigl(\xi_3\bigr) - L(\xi_1) + \rho - \tfrac{1}{3} \leqslant L(f(\xi_2)) + \rho - \tfrac{1}{3} - L(\xi_1) = L(\xi_2) + 2\rho - \tfrac{2}{3} - L(\xi_1) \leqslant L(\xi_2) + 3 \rho -1 - L(f(\xi_1)) \leqslant 3\rho - 1$. 
	
		\smallskip 
	\smallskip 
	\smallskip 
	\smallskip 

	In both cases, the variation of the \emph{code function} between the two adjacent \emph{red states} $R$ and $S$ is bounded by $3\rho - 1$.  
	Hence, there exist $r \in R$ and $s \in S$ satisfying $L(r) - L(s) \leqslant 3\rho - 1$. This completes the proof.
\end{proof}

\begin{theorem}\label{movement:red:black}
	Let $b \in P$ be any black point. Then there exists a red point $r \in P$ such that $	L(b) - L(r) \leqslant 4 \rho - \tfrac{5}{3}$, where $\rho$ denotes the rotation number of the triod–twist cycle $P$.
\end{theorem}

\begin{proof}
	Since $\rho > \tfrac{1}{3}$, the cycle $P$ contains at least one \emph{red} point.  
	Let $r \in P$ be a \emph{red} point, and assume without loss of generality that $r \in b_0$, where $b_0, b_1, b_2$ denote the \emph{branches} of $Y$ in \emph{canonical order}.  
	We consider three possible positions of the \emph{black point} $b$.

		\smallskip 
	\smallskip 
	\smallskip 
	\smallskip 

	Case 1: $b \in b_0$.  	If $r \geqslant b$, then $L(b) - L(r) \leqslant 0$.  	Assume instead that $b>r$.  
	Since the collections $\{t \in P : r > t\}$ and $\{t \in P : t \geqslant b\}$ are not invariant under $f$, there exist a  \emph{black point} $\xi$ and a \emph{red point}  $\eta \in P$ such that $ r > \xi$, $f(\xi) \geqslant \eta$ and $f(\eta) \geqslant b$.  	$L(b) - L(r) < L(f(\eta)) - L(\xi)  = L(f(\eta)) -\{ L(f(\xi)) - \rho + \frac{1}{3} \} = L(f(\eta)) - L(f(\xi)) + \rho - \frac{1}{3} = L(\eta) + \rho - \frac{2}{3} - L(f(\xi)) + \rho - \frac{1}{3} 	\leqslant 2\rho - 1$.

	\smallskip 
		\smallskip 
			\smallskip 
				\smallskip 
	Case 2: $b \in b_1$.  	Since the sets $\{t \in P : f(r) \geqslant t\}$ and $\{t \in P : t \geqslant b\}$ are not invariant under $f$, one of the following two subcases must occur.

	\emph{Subcase 2.1:} There exist \emph{black points} $\xi, \eta \in P$ such that 
	$f(r) \geqslant \xi$, $f(\xi) \geqslant \eta$, and $f(\eta) \geqslant b$.  
	Then, $L(b) - L(r) \leqslant L\bigl(f(\eta)\bigr) - \bigl( L(f(r)) - \rho + \tfrac{2}{3} \bigr) = L\bigl(f(\eta)\bigr) - L(f(r)) + \rho - \tfrac{2}{3} = L(\eta) + \rho - \tfrac{1}{3} - L(f(r)) + \rho - \tfrac{2}{3} = L\bigl(\eta\bigr) - L(f(r)) + 2\rho - 1 \leqslant L\bigl(f(\xi)\bigr) - L(f(r)) + 2\rho - 1 = L(\xi) + \rho - \tfrac{1}{3} - L(f(r))  + 2\rho - 1 \leqslant 3\rho - \tfrac{4}{3}$.

	\emph{Subcase 2.2:} There exists a \emph{red point} $\xi \in P$ such that 
	$f(r) \geqslant \xi$ and $f(\xi) \geqslant b$.  In this situation, we have $L(b) - L(r) \leqslant L\bigl(f(\xi)\bigr) - \bigl( L(f(r)) - \rho + \tfrac{2}{3} \bigr) = L(\xi) + \rho - \tfrac{2}{3} - L(f(r)) + \rho - \tfrac{2}{3} \leqslant 2\rho - \tfrac{4}{3}$.

\smallskip 
\smallskip 
\smallskip 
\smallskip 

Case 3: $b \in b_2$. If $f(r) \geqslant b$, then $L(b) - L(r) = L(b) - L(f(r)) + \rho - \frac{2}{3} = L(b) - L(f(r)) + \rho - \frac{2}{3} \leqslant \rho - \frac{2}{3}$.  Now, let $b \geqslant f(r)$. Since, the sets $\{ t \in P: t \geqslant b\}$ and $\{ t \in P:  f(r) \geqslant t\}$ are not invariant under $f$, we have three sub-cases:

 \emph{Sub-case 3.1:} There exists \emph{black points} $\xi_1, \xi_2$ and $\xi_3$ such that $f(r) \geqslant \xi_1, f(\xi_1) \geqslant \xi_2$, $f(\xi_2) \geqslant  \xi_3$ and $f(\xi_3) \geqslant b$. Then, $L(b) - L(r) = L(b) - \{ L(f(r)) - \rho + \frac{2}{3}\} = L(b) - L(f(r)) + \rho - \frac{2}{3} \leqslant L(b) - L(\xi_1) + \rho -\frac{2}{3} = L(b) - L(f(\xi_1)) + 2 \rho -1 \leqslant  L(b ) - L(\xi_2) + 2 \rho -1 = L(b) - L(f(\xi_2)) + 3 \rho - \frac{4}{3} \leqslant L(b ) - L(\xi_3) + 3 \rho - \frac{4}{3} = L(b) - L(f(\xi_3)) + 4 \rho - \frac{5}{3} \leqslant 4 \rho - \frac{5}{3}$

 \emph{Sub-case 3.2} There exists a \emph{black point} $\xi_1$ and a \emph{red point} $\eta$ such that $f(r) \geqslant \xi_1$, $r \geqslant f(\xi_1)$, $f(\xi_1) \geqslant \eta$ and $f(\eta) \geqslant b$. Then $L(b) - L(r) = L(b) - \{ L(f(r)) - \rho + \frac{2}{3}\} = L(b) - L(f(r)) + \rho - \frac{2}{3} \leqslant L(b) - L(\xi_1) + \rho - \frac{2}{3} = L(b) - \{ L(f(\xi_1)) - \rho + \frac{1}{3} \}  + \rho - \frac{2}{3} \leqslant L(b) - L(\eta) + 2 \rho -1 = L(b) - L(f(\eta)) + 3 \rho - \frac{5}{3} \leqslant 3 \rho - \frac{5}{3}$
 
 \emph{Sub-case 3.3} There exists a \emph{red point} $\eta$ and  a \emph{black point} $\xi_2$ such that $f(r) \geqslant \eta$, $f(\eta) \geqslant \xi_2$ and $f(\xi_2) \geqslant  b$. Now, $L(b) - L(r) = L(b) - L(f(r)) + \rho - \frac{2}{3} \leqslant L(b) - L(\eta) + \rho - \frac{2}{3} = L(b) - L(f(\eta)) + 2 \rho - \frac{4}{3} \leqslant L(b) - L(f(\xi_2)) + 3 \rho - \frac{5}{3} \leqslant 3 \rho - \frac{5}{3}$

	\smallskip
	Among all possible cases, the maximal bound is attained in Subcase 3.1, giving  
	$L(b) - L(r) \leqslant 4\rho - \tfrac{5}{3}$.  	This completes the proof.
\end{proof}

\begin{theorem}\label{movement:black:red}
	Let $b \in P$ be any black point. Then there exists a red point $r \in P$ such that 
	$	L(r) - L(b) \leqslant 5\rho - \tfrac{5}{3}$,	where $\rho$ is the rotation number of the triod–twist cycle $P$.
\end{theorem}

\begin{proof}
	Since $\rho >  \tfrac{1}{3}$, the cycle $P$ must contain at least one \emph{red} point.  
	Choose any \emph{red point} $r \in P$, and assume without loss of generality that $r \in b_0$, where $b_0, b_1, b_2$ denote the \emph{branches} of $Y$ in canonical order.  
	We now determine the desired bound by examining the possible positions of the \emph{black point} $b$.
	
	\smallskip 
\smallskip 
\smallskip 

Case 1: $b \in b_0$.  
If $b \geqslant r$, then clearly $L(r) - L(b) \leqslant 0$, so we may assume that $r \geqslant b$. Observe that the sets $\{t \in P : t \geqslant r\}$ and $\{t \in P : f(b) \geqslant t\}$ are not invariant under $f$. Consequently, one of the following sub-cases must occur: either there exists a \emph{red point} $\eta \in P$ such that $f(b) \geqslant \eta$ and $f(\eta) \geqslant r$, or there exist \emph{black points} $\xi_1, \xi_2 \in P$ satisfying $f(b) \geqslant \xi_1$, $f(\xi_1) \geqslant \xi_2$, and $f(\xi_2) \geqslant r$.  Since $P$ is \emph{regular}, it cannot \emph{force} a \emph{primitive cycle} of period $2$, and hence the first sub-case is ruled out. Moreover, for the same reason, in the second sub-case we must necessarily have $f(\xi_1) \geqslant \xi_2 > f(r)$.  By the definition of the \emph{code function}, we obtain,  $L(r) - L(b) \leqslant L(f(\xi_2))  - L(b) + \rho - \frac{1}{3} \leqslant  L(\xi_2) -  L(\xi_1) +2 \rho - \tfrac{2}{3} \leqslant  L(\xi_2) -  L(f(\xi_1)) +3\rho - 1 \leqslant 3 \rho -1$. 

	\smallskip 
	\smallskip 
	\smallskip

	Case 2: $b \in b_1$.  	If $f(b) \geqslant f(r)$, then $L(r) - L(b) = L(f(r)) - \rho + \tfrac{2}{3} - \bigl( L(f(b)) - \rho + \tfrac{1}{3} \bigr)	< \tfrac{1}{3}$.  Hence assume $f(r) \geqslant f(b)$.  	Since the sets $\{t \in P : b \geqslant t\}$ and $\{t \in P : t \geqslant r\}$ are not invariant under $f$, one of the following subcases must occur.

	\emph{Subcase 2.1:} There exist \emph{black points} $\xi_1, \xi_2, \xi_3, \xi_4 \in P$ with  
	$f(b) \geqslant \xi_1$, $f(\xi_1) \geqslant \xi_2$, $f(\xi_2) \geqslant \xi_3$, $f(\xi_3) \geqslant \xi_4$, and $f(\xi_4) \geqslant r$.  
	Then, $L(r) - L(b) \leqslant L(r) - \bigl( L(f(b)) - \rho + \tfrac{1}{3} \bigr) \leqslant L(r) - L(\xi_1) + \rho - \tfrac{1}{3} = L(r) - L(f(\xi_1)) + 2\rho - \tfrac{2}{3} \leqslant L(r) - L(f(\xi_2)) + 3\rho - 1 \leqslant L(r) - L(f(\xi_3)) + 4\rho - \tfrac{4}{3} \leqslant L(r) - L(f(\xi_4)) + 5\rho - \tfrac{5}{3} \leqslant 5\rho - \tfrac{5}{3}$.

	\emph{Subcase 2.2:} There exists a \emph{red point} $\xi$ with  
	$b \geqslant \xi$ and $f(\xi) \geqslant r$.  Then	$L(r) - L(b)	\leqslant L(r) - L(\xi) 	= L(r) - \bigl( L(f(\xi)) - \rho + \tfrac{2}{3} \bigr) 	\leqslant \rho - \tfrac{2}{3}$. 
	
	\smallskip 
\smallskip 
\smallskip

	Case 3: $b \in b_2$.  	If $b \geqslant f(r)$, then $L(f(r)) - L(b) \leqslant 0$, which yields, $L(r) + \rho - \frac{2}{3} - L(b) \leqslant 0$ and hence, $L(r) - L(b) \leqslant \frac{2}{3} - \rho $.  So assume $f(r) > b$.  Also, since , $P$ is \emph{regular}, we have $ r > f(b)$. Since neither $\{t \in P : b > t\}$ nor $\{t \in P : t > r\}$ is invariant, one of the following holds.
	
		\emph{Subcase 3.1:} There exist \emph{black points} $\xi_1, \xi_2, \xi_3 \in P$ with  $f(b) \geqslant \xi_1$, 
	$f(\xi_1) \geqslant \xi_2$, $f(\xi_2) \geqslant  \xi_3$, and $f(\xi_3)\geqslant r$ (See Figure \ref{drawing r10}). 	Then $	L(r) - L(b) \leqslant L(r) - \bigl( L(f(b)) - \rho + \tfrac{1}{3} \bigr) \le L(r) - L(\xi_1) + \rho - \tfrac{1}{3} = L(r) - L(f(\xi_1)) + 2\rho - \tfrac{2}{3} \le L(r) - L(f(\xi_2)) + 3\rho - 1 \leqslant L(r) - L(f(\xi_3)) + 4\rho - \tfrac{4}{3} \leqslant 4\rho - \tfrac{4}{3}$.

\begin{figure}[H]
	\centering
	\includegraphics[width=0.4\textwidth]{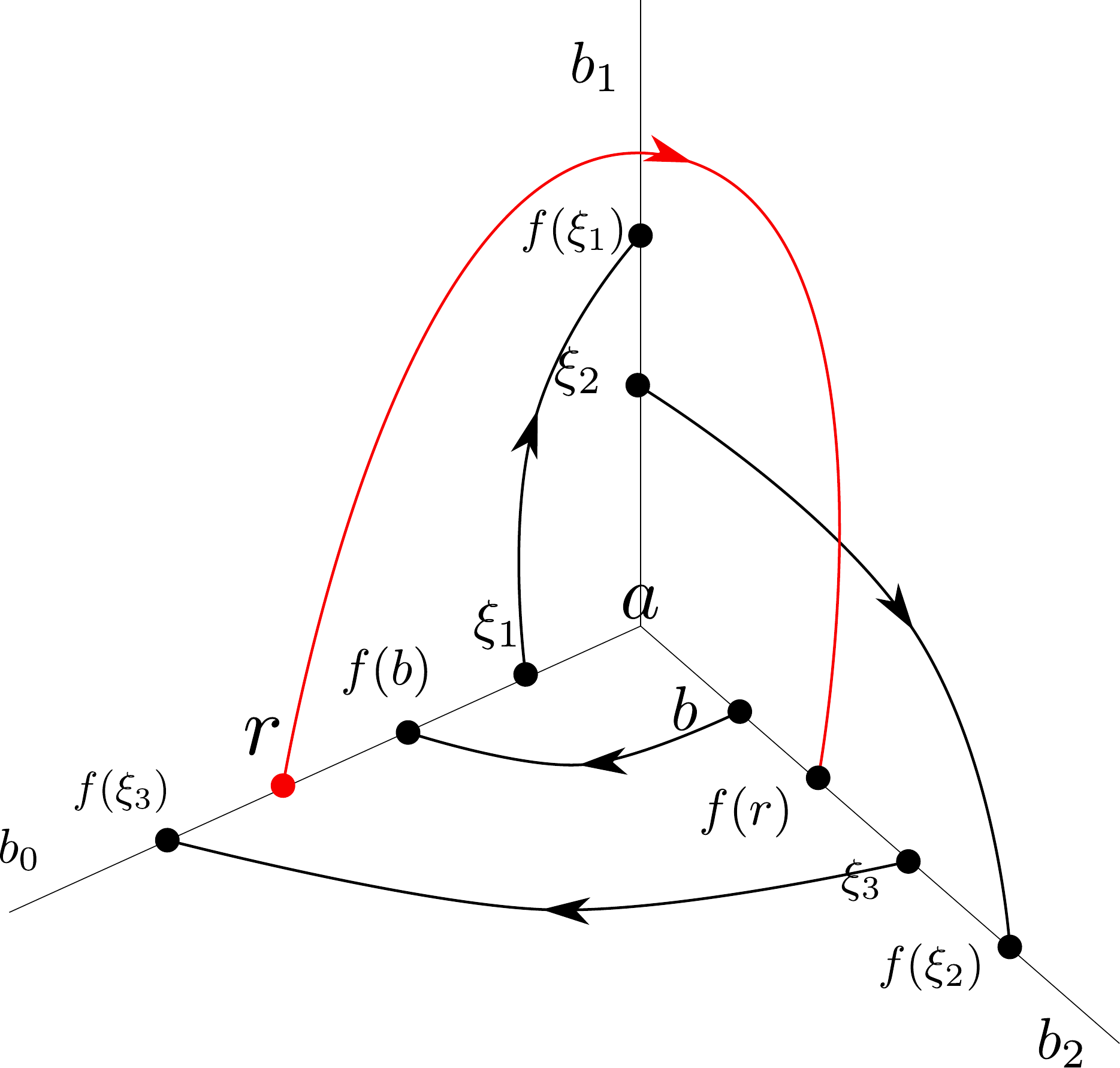}
	\caption{}
	\label{drawing r10}
\end{figure}

	\emph{Subcase 3.2:} There exist a \emph{black point} $\xi$ and a \emph{red point} $\eta$ with  
$f(b) \geqslant \xi$, $f(\xi) \geqslant  \eta$, and $f(\eta)\geqslant r$.  Then $L(r) - L(b)= L(r) - \bigl( L(f(b)) - \rho + \tfrac{1}{3} \bigr) 	\leqslant L(r) - L(\xi) + \rho - \tfrac{1}{3} \leqslant L(r) - L(\eta) + 2\rho - \tfrac{2}{3} =  L(r) - L(f(\eta)) + 3 \rho - \frac{4}{3} \leqslant 3\rho - \tfrac{4}{3}$.

	\smallskip 
	\smallskip 
	\smallskip

	Among all cases, the largest bound arises in Subcase~2.1.  Thus, for every \emph{black point} $b \in P$, there is a \emph{red point} $r \in P$ with	$L(r) - L(b) \leqslant 5\rho - \tfrac{5}{3}$. This completes the proof.
\end{proof}

\begin{theorem}\label{chi:bound:red}
	Let $P$ be a triod–twist cycle of modality $m$ and rotation number $\rho > \tfrac{1}{3}$.  
	Then the total oscillation of the code function on $P$ satisfies	$\chi(P) <  m + 3$. 
\end{theorem}

\begin{proof}
	According to Theorems~\ref{movement:black:red} and \ref{movement:red:black}, for every \emph{black point} $b \in P$, there exists \emph{red points} $r_1, r_2 \in P$ such that both 	$L(b) - L(r_1) \leqslant 4\rho - \tfrac{5}{3} \quad \text{and} \quad 	L(r_2) - L(b) \leqslant 5\rho - \tfrac{5}{3}$. 
	Hence, each \emph{black point} lies within $5\rho - \tfrac{5}{3}$ (in \emph{code}) of some \emph{red point}, both above and below.  	Consequently, the total \emph{oscillation} of the \emph{code function} satisfies $\chi(P) \leqslant  \chi(R(P)) + 5\rho - \tfrac{5}{3}$, where $R(P)$ denotes the set of all \emph{red points} of $P$.
	
	Let $\mathcal{R}(P) = \{S_1, S_2, \ldots, S_\ell\}$ be the collection of all \emph{red states} of $P$.  
	By Theorem~\ref{red:state:phase:bound}, each \emph{red state} $S_i$ satisfies $\chi(S_i) \leqslant  3\rho - 1$. Moreover, by Theorems~\ref{movement:b0:b2} and \ref{red:adjacent:states}, in the transitions between consecutive \emph{red states}, the variation in \emph{code} is bounded by $5\rho - \tfrac{5}{3}$.  	Since at most $\ell - 1$ such inter–state transitions can occur, it follows that $\chi(R(P)) \leqslant \sum_{i=1}^{\ell} \chi(S_i) + (\ell - 1)\Big(5\rho - \tfrac{5}{3}\Big) \leqslant  \ell(3\rho - 1) + (\ell - 1)\Big(5\rho - \tfrac{5}{3}\Big) = (8\ell - 5)\rho + \tfrac{-8\ell + 5}{3}$. 
	
	Substituting this estimate into the previous inequality gives $	\chi(P) \leqslant (8\ell - 5)\rho + \tfrac{-8\ell + 5}{3} + (5\rho - \tfrac{5}{3}) 	= 8\ell\rho - \tfrac{8\ell}{3} = 8 \ell (\rho - \frac{1}{3}).$ 	Since,  $\rho \leqslant \tfrac{1}{2}$ by Theorem~\ref{rho:less:half}, we have $\chi(P) \leqslant 8 \ell (\frac{1}{2} - \frac{1}{3})
	\leqslant \tfrac{4}{3}\ell$ 
	
	By Lemma \ref{first:lemma},  the number of \emph{red states} $\ell$ is at most $\frac{m}{2} + 2$, it follows that $\chi(P) < \tfrac{4}{3}(\frac{m}{2} + 2) <  \frac{2m}{3}+ \frac{8}{3} <  m+3$. Thus, the desired bound $\chi(P) < m + 3$ holds for all \emph{triod–twist} cycles with $\rho > \tfrac{1}{3}$.
\end{proof}

\subsection*{3.3. Conjugacy with Circle Rotations}

We now establish the connection between \emph{triod–twist cycles} and
periodic orbits of circle rotations. Let $P$ be a \emph{triod–twist} cycle with \emph{rotation number} $\rho=\tfrac pq$, where $p$ and $q$ are co-prime, and let $f$ be its associated $P$-\emph{linear} map.
Our aim is to construct a conjugacy between the dynamics of $f$ on $P$
and the \emph{rotation} $R_\rho\colon S^1\to S^1$, $R_\rho(x)=x+\rho\pmod1$, restricted on one of its cycles. 

Cutting the unit circle at one point and identifying it with $[0,1)$,
we obtain the map $g(x)=x+\tfrac pq\pmod1,  x\in[0,1)$.  Let $Q$ be the orbit of $0$ under $g$. Then $P$ and $Q$ are both $q$-cycles, and we can define a bijection $\psi:P\to Q$ preserving the cyclic order induced by $f$ and $g$, as follows.  To make $\psi$ explicit, define the \emph{code function} $L:P\to\mathbb R$, normalize it so that $L(b)=0$ for the point $b\in P$ with minimal code,
and set $\psi(b)=0,\qquad \psi(f(x))=g(\psi(x))\ \text{for all }x\in P$. If $\psi(x)=L(x)\pmod1$, then
$L(f(x))=L(x)+\tfrac pq \equiv\psi(x)+\tfrac pq \equiv g(\psi(x)) =\psi(f(x))\pmod1$,
and since $L(b)=\psi(b)=0$, induction gives $\psi(z)\equiv L(z)\pmod1$ for every $z\in P$. Thus $\psi$ conjugates $f|_P$ with $g|_Q$. Thus, we obtain the following result. 

\begin{theorem}\label{thm:triod-twist-rotation}
	Let $P$ be a triod–twist cycle of modality $m$ and rotation number
	$\rho=\tfrac pq$ with $g.c.d (p,q)=1$. Let $g: S^1 \to S^1$ be the map defined by $g(x)=x+\tfrac pq\pmod1, x\in[0,1)$ and let $Q$ be the orbit of $0$ under $g$. Then the conjugacy $\psi:P\to Q$ between $P$ and $Q$ is piece-wise monotone with at most $m+3$ laps. 

\end{theorem}

\begin{proof}
	By Theorems \ref{thm:chi-bound-final} and \ref{chi:bound:red}, 
	the total \emph{oscillation} of the \emph{code function} $L$ is bounded above by $m+3$. Also, from Theorem \ref{tri-od:rot:twist:order:inv}, $L(x)$ is \emph{monotone} on each collection of consecutive points of $P$ lying on a \emph{branch} of $Y$, for which the \emph{integral part} of $L(x)$ is same.  Now, since $\psi(x)\equiv L(x)\pmod1$, so, $\psi(x)$ is \emph{piece-wise monotone} with at-most $m+3$ \emph{laps}. 

\end{proof}

\end{document}